\documentclass[11pt]{article}
\usepackage{amssymb}
\usepackage{amsmath}
\usepackage{amsthm}
\usepackage{amsfonts}
\usepackage{cite,color}
\usepackage{graphics}
\usepackage{graphicx}
\usepackage{epsfig}
\usepackage{epstopdf}
\usepackage{subfigure}
\usepackage[top=1.00 in,bottom=1.00 in,left=1.00 in,right=1.00 in]{geometry}
\usepackage{setspace}
\usepackage{float}
\usepackage{geometry}
\usepackage{caption}
\usepackage{mathrsfs}
\usepackage{subcaption} 
\usepackage{booktabs}
\usepackage{multirow}
\usepackage{siunitx}
\usepackage{enumerate}

\newtheorem{theorem}{Theorem}

\newtheorem{remark}{Remark}
\newtheorem{lemma}{Lemma}

\newtheorem{definition}{Definition}

\numberwithin{equation}{section}
\numberwithin{theorem}{section}
\numberwithin{definition}{section}
\numberwithin{lemma}{section}
\numberwithin{corollary}{section}
\numberwithin{remark}{section}
\numberwithin{example}{section}
\numberwithin{Claim}{section}

\newcommand\blfootnote[1]{
  \begingroup
  \renewcommand\thefootnote{}
  \footnote{#1}
  \addtocounter{footnote}{-1}
  \endgroup}

\title{\vspace{-10mm}Matrix-Free Stabilized BDF Schemes for Semilinear Parabolic Equations with Unconditional Maximum Bound Principle Preservation and Energy Stability}
\author{Haishen Dai \and Huan Lei \and Bin Zheng}
\date{} %

\begin{document}
\maketitle
\blfootnote{Haishen Dai, Department of Computational Mathematics, Science \& Engineering, Michigan State University, East Lansing, MI 48824, USA. \texttt{daihaish@msu.edu}.}
\blfootnote{Huan Lei, Department of Computational Mathematics, Science \& Engineering, Michigan State University, East Lansing, MI 48824, USA, and Department of Statistics \& Probability, Michigan State University, East Lansing, MI 48824, USA. \texttt{leihuan@msu.edu}.}
\blfootnote{Bin Zheng, School of Mathematics, Statistics and Mechanics, Beijing University of Technology, Beijing 100124, China. \texttt{zhengbin@bjut.edu.cn}.}

\date{}
\maketitle
\begin{abstract}
Abstract: We develop a family of stabilized backward differentiation formula (sBDF) schemes of orders one through four for semilinear parabolic equations. The proposed methods are designed to achieve three properties that are rarely available simultaneously in high-order time discretizations: unconditional preservation of the maximum bound principle (MBP), unconditional discrete energy stability, and practical matrix-free implementation. The construction integrates carefully designed stabilization terms, fixed-point iterations, and a pointwise cut-off strategy. The nonlinear algebraic systems arising from the implicit sBDF discretizations are solved by fixed-point iteration, resulting in fully matrix-free algorithms. This makes the approach particularly attractive for practical computations on general domains and under mixed boundary conditions, where FFT-based exponential time differencing methods are often unavailable or inefficient. We further present a unified analysis for the fully implemented schemes, explicitly incorporating the interplay among time discretization, nonlinear iteration, and cut-off. Unconditional contractivity of the fixed-point iterations and error estimates are established. For the Allen–Cahn equation, we additionally prove an unconditional discrete energy dissipation law. Numerical experiments confirm the theoretical convergence rates and demonstrate the robustness and efficiency of the proposed methods, particularly relative to ETD-based approaches for problems with mixed boundary conditions.
\\[2mm]
\\[2mm]
\noindent\end {abstract}

\section{Introduction}
We consider numerical approximation of semilinear parabolic equations of the form
\begin{equation}\label{eq:stiff equation}
\phi_t = \alpha \Delta \phi + f(\phi), \qquad (\mathbf{x},t)\in \Omega \times (0,T),
\end{equation}
subject to suitable initial and boundary conditions, where $\alpha>0$ is the diffusion coefficient and $f(\phi)$ is a smooth nonlinear term. Problems of this type arise in a broad range of physical and biological applications. A prototypical example is the Allen--Cahn equation~\cite{Allen1979,Evans1992}, which plays a fundamental role in diffuse-interface descriptions of phase separation and has been widely used in materials science~\cite{Heida2012}, image processing~\cite{Benes2004,Lee2019}, and interfacial motion~\cite{Aihara2019}. Related semilinear reaction--diffusion systems also appear in biological applications such as prostate cancer growth modeling~\cite{Colli2020,Huang2024}.

Two structural features are central to the reliable simulation of such problems. The first is the maximum bound principle (MBP), which constrains the solution to remain within physically meaningful bounds determined by the initial and boundary data. The second is energy dissipation: many phase-field models possess an underlying energy functional that decreases monotonically in time. At the discrete level, the preservation of these properties is often critical for avoiding spurious oscillations, nonphysical solutions, and qualitatively incorrect long-time dynamics. However, standard time discretizations for parabolic PDEs~\cite{Thomee2008,Feng2003,Zhang2009} generally fail to preserve either the MBP or a compatible discrete energy law. As emphasized in~\cite{Ju2021}, loss of the MBP may lead to fundamentally incorrect numerical simulations. This makes the design of high-order time discretization methods that preserve both properties a challenging task.

A substantial literature has been devoted to structure-preserving discretizations. For MBP-preserving schemes, we refer to
\cite{Hou2017,Hou2020,Hou2020maxprinciple,Ju2021,Nan2022,Liu2014MPDG,Du2021,Fu2022,Liao2020BDF,Gong2023,Zhu2025}
and the references therein. For unconditionally energy-stable methods for gradient flows, representative approaches include convex-splitting schemes~\cite{Elliott1993,Eyre1998}, stabilized methods~\cite{Zhu1999,Shen2010}, and the IEQ/SAV frameworks~\cite{Yang2016homopolymer,Shen2018SAV}. Despite these developments, existing methods typically enforce MBP preservation and energy stability separately, or only under time-step restrictions, or at low temporal order. The simultaneous realization of \emph{high-order accuracy}, \emph{unconditional MBP preservation}, and \emph{unconditional discrete energy stability} remains rare. Achieving all three within a computationally efficient framework is even more difficult.

Among existing MBP-preserving approaches, the cut-off strategy introduced in~\cite{Lu2013Cutoff} is particularly attractive because of its simplicity and robustness. Building on this idea, Li et al.~\cite{Li2020Cutoff} proposed arbitrarily high-order exponential time differencing (ETD) schemes with cut-off for the Allen--Cahn equation, obtaining unconditional MBP preservation together with high-order temporal accuracy. ETD-type methods, including~\cite{Li2020Cutoff,Du2021MBP}, are especially effective on simple domains with periodic boundary conditions, where FFT-based diagonalization enables highly efficient matrix-free implementations. Their efficiency, however, deteriorates significantly on irregular domains or under mixed boundary conditions. In such settings, the discrete spatial operator typically lacks the structure required for FFT acceleration, and ETD methods must rely on large sparse-matrix representations together with the evaluation of matrix exponential-related operators via Krylov subspace approximations~\cite{Gallopoulos1992,Hochbruck1997,Hochbruck1998,Qiu2019} or Pad\'e-type approximations~\cite{Ehle1973,khaliq2009,Dai2023,Dai2026}. These procedures are effective but can become prohibitively expensive as the spatial resolution increases.

This paper addresses precisely this gap. Our goal is to construct high-order time discretizations for semilinear parabolic equations that simultaneously satisfy the following three requirements: unconditional MBP preservation, unconditional discrete energy stability, and practical matrix-free implementation on general domains and under mixed boundary conditions. To achieve this, we develop a family of stabilized backward differentiation formula schemes, denoted sBDF$k$ for $k=1,2,3,4$. The starting point is the fully implicit BDF discretization. We then introduce carefully designed stabilization terms and solve the resulting nonlinear algebraic systems by fixed-point iteration, thereby converting the implicit schemes into fully matrix-free algorithms. A simple pointwise cut-off is incorporated at the algorithmic level to enforce the MBP. The stabilization is designed so that the nominal temporal order of the underlying BDF formulas is retained, while the associated fixed-point mappings become unconditionally contractive. This contractive structure is a key ingredient in the analysis and in the practical efficiency of the methods.

The contribution of this work is twofold. First, we propose a family of high-order stabilized BDF schemes that combines unconditional MBP preservation and unconditional discrete energy stability within a single unified framework. To the best of our knowledge, such a combination is not available for high-order BDF-type discretizations in a fully matrix-free setting. Second, we develop a unified analysis at the level of the final implemented algorithms, rather than only for the formal implicit schemes. In particular, the contractivity of the fixed-point iteration, the accuracy of the discretization, the MBP-preserving property, and the energy dissipation law are analyzed in a framework that explicitly incorporates the interplay among time discretization, nonlinear iteration, and cut-off. For the Allen--Cahn equation, we prove that the practical after-cut-off iterates satisfy an unconditional discrete energy dissipation law. This algorithm-level analysis is essential, since it directly justifies the stability and robustness of the implemented methods.

The proposed approach offers a compelling alternative to ETD-based structure-preserving schemes when FFT-based diagonalization is unavailable or inefficient. Because the methods are matrix-free, they are straightforward to implement and attractive for large-scale computations. At the same time, they retain strong stability properties that are typically difficult to reconcile with high-order time discretization.

The remainder of this paper is organized as follows.
Section~\ref{sec:fully discrete} presents the fully discrete sBDF schemes and their unified formulation.
Section~\ref{sec:convergence of iteration} proves unconditional contractivity of the fixed-point iterations.
Section~\ref{sec:error_analysis} establishes the error analysis.
Section~\ref{sec:MBP} proves unconditional MBP preservation of the proposed algorithms.
Section~\ref{sec:energy} presents the discrete energy stability analysis for the Allen--Cahn equation.
Section~\ref{sec:numerical experiments} reports numerical experiments verifying the convergence rates and demonstrating the efficiency of the proposed matrix-free schemes, including comparisons with ETD-based methods under mixed boundary conditions.
Finally, Section~\ref{sec:conclusion} concludes the paper.

\section{Discretization and fixed-point iteration}\label{sec:fully discrete}

The nonlinear term $f$ in equation \eqref{eq:stiff equation} is assumed to satisfy the following standard conditions
(see, e.g., \cite{Du2021MBP}): there exist constants $\beta>0$ and $B>0$ such that
\begin{equation}\label{eq:f_assump}
|f'(\xi)|\le B,\qquad \forall\,\xi\in[-\beta,\beta],
\qquad
f(\beta)\le 0,\qquad f(-\beta)\ge 0.
\end{equation}
These assumptions are commonly used to ensure the validity of the maximum bound principle (MBP)
and are satisfied by a broad class of nonlinearities arising in phase-field models.

To facilitate the construction of numerical schemes that preserve the MBP unconditionally,
we first rewrite equation \eqref{eq:stiff equation} in the equivalent form
\begin{equation}\label{eq:modified_pde}
\begin{aligned}
\frac{\partial \phi}{\partial t} = \alpha \Delta \phi - B\phi + \mathcal N(\phi),
\end{aligned}
\end{equation}
where the modified nonlinear term is defined by
\begin{equation}\label{eq:N_def}
\mathcal N(\phi):=f(\phi)+B\phi .
\end{equation}
Under \eqref{eq:f_assump}, the mapping $\mathcal N$ is Lipschitz continuous and uniformly bounded
on $[-\beta,\beta]$:
\begin{equation}\label{eq:N_props}
|\mathcal N(\phi_1)-\mathcal N(\phi_2)|
\le 2B\,|\phi_1-\phi_2|,
\qquad \forall\,\phi_1,\phi_2\in[-\beta,\beta],
\end{equation}
and
\begin{equation}\label{eq:N_bound}
|\mathcal N(\phi)|\le B\beta,
\qquad \forall\,\phi\in[-\beta,\beta].
\end{equation}

Let $\|\cdot\|$ denote the discrete $\ell^2$-norm on the spatial grid, i.e.,
$\|v\|^2 := h^2\sum_{i,j} |v_{i,j}|^2$.
For any grid function $v$, denote by $(\Delta_h v)_{i,j}$ the standard five-point
finite-difference Laplacian. 

All numerical schemes developed in this paper are based on the reformulated system \eqref{eq:modified_pde}.
For simplicity, we set $\beta=1$ throughout. We discretize the spatial operator using a finite
difference method on a uniform Cartesian grid. In particular, the Laplacian $\Delta$ is approximated
by the second-order central difference operator $\Delta_h$, which leads to the semi-discrete system
of ordinary differential equations (ODEs)
\begin{equation}\label{eq:semi-discrete}
\frac{\partial \phi}{\partial t} = \alpha \Delta_h \phi - B \phi + \mathcal{N}(\phi).
\end{equation}

\subsection{sBDF1 scheme}\label{subsec:sBDF1}

We introduce the stabilized first-order backward differentiation formula  (sBDF1) to discretize \eqref{eq:semi-discrete} in time. Let $\widetilde{\phi}^n_{i,j}$ denote the approximation of the exact solution $\phi(x_i,y_j,t_n)$. The method is obtained by adding a stabilization term $B(\widetilde{\phi}_{i,j}^{n+1}-\widetilde{\phi}_{i,j}^{n})\Delta t$ to the standard backward Euler scheme, resulting in the following discrete problem:
\begin{equation}\label{eq:sBDF1_implicit}
\begin{aligned}
\Bigl(1+\frac{4\alpha\Delta t}{h^{2}}+2B\Delta t\Bigr)\widetilde{\phi}_{i,j}^{\,n+1}
&=
(1+B\Delta t)\widetilde{\phi}_{i,j}^{n} \\
&\quad
+\frac{\alpha\Delta t}{h^{2}}
\Bigl(
\widetilde{\phi}_{i+1,j}^{\,n+1}
+\widetilde{\phi}_{i-1,j}^{\,n+1}
+\widetilde{\phi}_{i,j-1}^{\,n+1}
+\widetilde{\phi}_{i,j+1}^{\,n+1}
\Bigr)
+\Delta t\,\mathcal N(\widetilde{\phi}_{i,j}^{\,n+1}).
\end{aligned}
\end{equation}

The nonlinear system \eqref{eq:sBDF1_implicit} is solved approximately by a fixed-point iteration, leading to a matrix-free implementation.
Starting from the initial guess $\widetilde{\phi}^{\,n+1,(0)}=\phi^n$, we define the auxiliary iterate
\begin{equation}\label{eq:fixed_point_iteration_aux}
\begin{aligned}
\widetilde{\phi}^{\,n+1,(m+1)}_{i,j}
=
\frac{1}{1+\frac{4\alpha\Delta t}{h^2}+2B\Delta t}
\Bigl[
&(1+B\Delta t)\phi^n_{i,j}
+\frac{\alpha\Delta t}{h^2}
\bigl(
\widetilde{\phi}^{\,n+1,(m)}_{i+1,j}
+\widetilde{\phi}^{\,n+1,(m)}_{i-1,j} \\
&\qquad
+\widetilde{\phi}^{\,n+1,(m)}_{i,j+1}
+\widetilde{\phi}^{\,n+1,(m)}_{i,j-1}
\bigr)
+\Delta t\,\mathcal N\bigl(\widetilde{\phi}^{\,n+1,(m)}_{i,j}\bigr)
\Bigr],
\qquad m=0,1,\dots 
\end{aligned}
\end{equation}

For the sBDF1 scheme, the resulting fixed-point mapping is monotone and preserves the discrete
maximum bound principle; hence no cut-off operator is needed. The iteration is terminated once
\[
\|\widetilde{\phi}^{\,n+1,(m+1)}-\widetilde{\phi}^{\,n+1,(m)}\|
< \varepsilon,\;\;
\text{where}\;\;
\varepsilon=C \min\{\Delta t^2, h^2\},
\]
and we define the numerical solution at time level $t_{n+1}$ as $\phi^{\,n+1} := \widetilde{\phi}^{\,n+1,(m_\star)}$ 
where $m_\star$ is the (finite) stopping index.

\subsection{sBDF2 scheme}

To obtain a second-order accurate temporal discretization, we 
start from the standard second-order backward differentiation formula, and introduce a stabilization term $2B\big(\widetilde{\phi}^{n+1}-2\widetilde{\phi}^{n}+\widetilde{\phi}^{n-1}\big)\Delta t$ where $\widetilde{\phi}^{\,n+1}$ denotes the solution of the stabilized BDF2 scheme (sBDF2):
\begin{equation}\label{eq:bdf2_stabilized}
\begin{aligned}
\left( 3 + \frac{8\alpha \Delta t}{h^2} + 4B \Delta t \right) \widetilde{\phi}^{\,n+1}_{i,j} 
=&\; (4 + 4B \Delta t) \widetilde{\phi}^{n}_{i,j} - (1 + 2B \Delta t) \widetilde{\phi}^{n-1}_{i,j} \\
&\; + \frac{2\alpha \Delta t}{h^2} \left(
\widetilde{\phi}^{\,n+1}_{i+1,j} + \widetilde{\phi}^{\,n+1}_{i-1,j}
+ \widetilde{\phi}^{\,n+1}_{i,j-1} + \widetilde{\phi}^{\,n+1}_{i,j+1}
\right)
+ 2 \Delta t \, \mathcal{N}(\widetilde{\phi}^{\,n+1}_{i,j}).
\end{aligned}
\end{equation}

The nonlinear system \eqref{eq:bdf2_stabilized} is solved by a fixed-point iteration.
Starting from $\widetilde{\phi}^{n+1, (0)} = \phi^n$, for $m=0,1,\dots$, we first compute the auxiliary iterate
$\widetilde{\phi}^{\,n+1,(m+1)}$:

\begin{equation}\label{eq:bdf2_fixed_point}
\begin{aligned}
\widetilde{\phi}^{\,n+1,(m+1)}_{i,j}
=\;& \frac{1}{\,3+\frac{8\alpha\Delta t}{h^2}+4B\Delta t\,}
\Bigl[
(4+4B\Delta t)\phi^{n}_{i,j}
-(1+2B\Delta t)\phi^{n-1}_{i,j} \\[-2pt]
&\qquad
+\frac{2\alpha\Delta t}{h^2}\Bigl(
\phi^{n+1,(m)}_{i+1,j}
+\phi^{n+1,(m)}_{i-1,j}
+\phi^{n+1,(m)}_{i,j-1}
+\phi^{n+1,(m)}_{i,j+1}
\Bigr)
+2\Delta t\,\mathcal N\!\bigl(\phi^{n+1,(m)}_{i,j}\bigr)
\Bigr].
\end{aligned}
\end{equation}

The $k$th-order BDF discretizations do not preserve the maximum bound principle in general for $k\geq 2$. 
We therefore apply a pointwise cut-off operator $\Pi_{[-1,1]}$ after each fixed-point update 
and define the algorithmic iterate
\begin{equation}
\phi^{n+1,(m+1)}_{i,j}
=\; \Pi_{[-1,1]}\!\left(\widetilde{\phi}^{\,n+1,(m+1)}_{i,j}\right)
= \max\!\left\{
\min\!\left\{\widetilde{\phi}^{\,n+1,(m+1)}_{i,j},\,1\right\},\, -1
\right\}.
\end{equation}

The fixed point iteration is terminated once
$$
\|\phi^{n+1,(m+1)}-\phi^{n+1,(m)}\|<\varepsilon,\;\;
\text{where}\;\;
\varepsilon=C \min\{{\Delta t}^3, h^2\},
$$
and the numerical solution at time level 
$t_{n+1}$ is defined by $\phi^{n+1}:=\phi^{\,n+1,(m_\ast)}$.

\subsection{sBDF3 scheme}

To obtain a third-order accurate temporal discretization, we start from the standard
third-order BDF and introduce a stabilization term
$6B(\widetilde{\phi}^{n+1}-3\widetilde{\phi}^n+3\widetilde{\phi}^{n-1}-\widetilde{\phi}^{n-2})\Delta t$,
where $\widetilde{\phi}^{\,n+1}$ denotes the solution of the stabilized BDF3 scheme (sBDF3):
\begin{equation}\label{eq:sBDF3}
\begin{aligned}
\Bigl(11+\frac{24\alpha\Delta t}{h^2}+12B\Delta t\Bigr)\widetilde{\phi}_{i,j}^{\,n+1}
&=
(18+18B\Delta t)\widetilde{\phi}_{i,j}^{n}
-(9+18B\Delta t)\widetilde{\phi}_{i,j}^{n-1}
+(2+6B\Delta t)\widetilde{\phi}_{i,j}^{n-2} \\
&\quad
+\frac{6\alpha\Delta t}{h^2}
\Bigl(
\widetilde{\phi}_{i+1,j}^{\,n+1}
+\widetilde{\phi}_{i-1,j}^{\,n+1}
+\widetilde{\phi}_{i,j-1}^{\,n+1}
+\widetilde{\phi}_{i,j+1}^{\,n+1}
\Bigr)
+6\Delta t\,\mathcal N(\widetilde{\phi}_{ij}^{\,n+1}).
\end{aligned}
\end{equation}

The nonlinear system \eqref{eq:sBDF3} is solved by a fixed-point iteration.
Starting from $\widetilde{\phi}^{\,n+1,(0)}=\phi^n$, for $m=0,1,\dots$, we first compute the
auxiliary iterate
\begin{equation}\label{eq:sBDF3_FPI}
\begin{aligned}
\widetilde{\phi}_{i,j}^{\,n+1,(m+1)}
=
&\frac{1}{11+\frac{24\alpha\Delta t}{h^2}+12B\Delta t}
\Bigl[
(18+18B\Delta t)\phi_{i,j}^{n}
-(9+18B\Delta t)\phi_{i,j}^{n-1}
+(2+6B\Delta t)\phi_{i,j}^{n-2} \\
&\quad
+\frac{6\alpha\Delta t}{h^2}
\Bigl(
\phi_{i+1,j}^{\,n+1,(m)}
+\phi_{i-1,j}^{\,n+1,(m)}
+\phi_{i,j-1}^{\,n+1,(m)}
+\phi_{i,j+1}^{\,n+1,(m)}
\Bigr)
+6\Delta t\,\mathcal N(\phi_{ij}^{\,n+1,(m)})
\Bigr].
\end{aligned}
\end{equation}

We apply a pointwise cut-off operator $\Pi_{[-1,1]}$ after each
fixed-point update and define the algorithmic iterate
\begin{equation}\label{eq:sBDF3_cutoff}
\phi_{i,j}^{\,n+1,(m+1)}
=
\Pi_{[-1,1]}\!\left(\widetilde{\phi}_{i,j}^{\,n+1,(m+1)}\right).
\end{equation}

The fixed-point iteration is terminated once
$$
\|\phi^{\,n+1,(m+1)}-\phi^{\,n+1,(m)}\|<\varepsilon\;\;
\text{where}\;\;
\varepsilon=C\min\{\Delta t^4,h^2\},
$$ 
and the numerical solution at time level $t_{n+1}$ is defined by
$
\phi^{n+1}:=\phi^{\,n+1,(m_\ast)},
$
where $m_\ast$ is the stopping index.

\subsection{sBDF4 scheme}\label{subsec:sBDF4}

Similarly, we start from the standard
fourth-order BDF and introduce a stabilization term
$12B(\widetilde{\phi}^{n+1}-4\widetilde{\phi}^n+6\widetilde{\phi}^{n-1}-4\widetilde{\phi}^{n-2}+\widetilde{\phi}^{n-3})\Delta t$,
to obtain the stabilized BDF4 scheme (sBDF4):
\begin{equation}\label{eq:sBDF4}
\begin{aligned}
\Bigl(25+ &\frac{48\alpha\Delta t}{h^2}+24B\Delta t\Bigr)\widetilde{\phi}^{\,n+1}_{i,j}
=
(48+48B\Delta t)\widetilde{\phi}^{n}_{i,j}
-(36+72B\Delta t)\widetilde{\phi}^{n-1}_{i,j} +(16+48B\Delta t)\widetilde{\phi}^{n-2}_{i,j} \\
&\quad
-(3+12B\Delta t)\widetilde{\phi}^{n-3}_{i,j} 
+\frac{12\alpha\Delta t}{h^2}
\Bigl(
\widetilde{\phi}^{\,n+1}_{i+1,j}
+\widetilde{\phi}^{\,n+1}_{i-1,j}
+\widetilde{\phi}^{\,n+1}_{i,j-1}
+\widetilde{\phi}^{\,n+1}_{i,j+1}
\Bigr)
+12\Delta t\,\mathcal N(\widetilde{\phi}^{\,n+1}_{ij}).
\end{aligned}
\end{equation}

The nonlinear system \eqref{eq:sBDF4} is solved by a fixed-point iteration.
Starting from $\widetilde{\phi}^{\,n+1,(0)}=\phi^n$, for $m=0,1,\dots$, we first compute the
auxiliary iterate
\begin{multline}\label{eq:sBDF4_FPI}
\widetilde{\phi}^{\,n+1,(m+1)}_{i,j}
=\frac{1}{25+\frac{48\alpha\Delta t}{h^2}+24B\Delta t}
\Bigl[
(48+48B\Delta t)\phi^{n}_{i,j}
-(36+72B\Delta t)\phi^{n-1}_{i,j} +(16+48B\Delta t)\phi^{n-2}_{i,j}\\
-(3+12B\Delta t)\phi^{n-3}_{i,j} 
+\frac{12\alpha\Delta t}{h^2}
\Bigl(
\phi^{\,n+1,(m)}_{i+1,j}
+\phi^{\,n+1,(m)}_{i-1,j}
+\phi^{\,n+1,(m)}_{i,j-1}
+\phi^{\,n+1,(m)}_{i,j+1}
\Bigr)
+12\Delta t\,\mathcal N(\phi^{\,n+1,(m)}_{ij})
\Bigr].
\end{multline}

We apply a pointwise cut-off operator $\Pi_{[-1,1]}$ after each
fixed-point update and define the algorithmic iterate
\begin{equation}\label{eq:sBDF4_cutoff}
\phi^{\,n+1,(m+1)}_{ij}
=
\Pi_{[-1,1]}\!\left(\tilde{\phi}^{\,n+1,(m+1)}_{i,j}\right).
\end{equation}

The fixed-point iteration is terminated once
$$
\|\phi^{\,n+1,(m+1)}-\phi^{\,n+1,(m)}\|<\varepsilon\;\;
\text{where}\;\;
\varepsilon=C\min\{\Delta t^5,h^2\},
$$ 
and the numerical solution at time level $t_{n+1}$ is defined by $\phi^{n+1}:=\phi^{\,n+1,(m_\ast)}$ where $m_\ast$ is the (finite) stopping index.

\begin{remark}
For all stabilized BDF schemes considered in this work, the stabilization term
can be written in the unified form
\[
\beta_k B\,\Delta t\,\nabla^k \phi^{n+1},
\qquad
\nabla^k \phi^{n+1}
=
\sum_{j=0}^{k}(-1)^j\binom{k}{j}\phi^{n+1-j},
\]
where $\nabla^k$ denotes the $k$-th order backward finite difference operator, 
$\beta_k>0$ is a scheme-dependent constant chosen to ensure unconditional
contractivity of the associated fixed-point iteration and unconditional discrete
energy stability. In particular, we take
$\beta_1=1$, $\beta_2=2$, $\beta_3=6$, and $\beta_4=12$.

Since $\nabla^k \phi^{n+1}=\mathcal O(\Delta t^k)$ for sufficiently smooth solutions,
the stabilization term is of order $\mathcal O(\Delta t^{k+1})$ and therefore does not
affect the $k$-th order temporal accuracy of the underlying BDF scheme.
\end{remark}

We emphasize that the stabilization term is independent of the cut-off operator used
to enforce the discrete maximum bound principle. In particular, for the first-order
scheme (sBDF1), the stabilized discretization itself preserves the maximum bound
principle, while for higher-order schemes (sBDF$k$, $k\ge2$), a pointwise cut-off
operator is incorporated to enforce unconditional MBP.

\subsection{Unified formulation for sBDF$k$ schemes}\label{subsec:sBDFk}

Let $D_k\phi(t_{n+1})$ denote the standard BDF$k$ approximation of $\partial_t\phi(t_{n+1})$:
\[
D_k\phi(t_{n+1})
:= \frac{1}{\Delta t}\sum_{\ell=0}^{k} a_{\ell,k}\phi(t_{n+1-\ell}),
\]
where $\{a_{\ell,k}\}_{l=0}^k$ are the BDF$k$ coefficients. In particular, when $k=1$, 
$a_{0,1}=1$, $a_{1,1}=-1$;
when $k=2$, $a_{0, 2}=3/2$, $a_{1, 2}=-2$, $a_{2,2}=1/2$;
when $k=3$, $a_{0, 3}=11/6$, $a_{1, 3}=-3$, $a_{2,3}=3/2$, $a_{3,3}=-1/3$;
when $k=4$, $a_{0, 4}=25/12$, $a_{1, 4}=-4$, $a_{2,4}=3$, $a_{3,4}=-4/3$, $a_{4,4}=1/4$.

With this notation, the implicit sBDF$k$ schemes can be expressed in the following unified form:
\begin{multline}
\left(
\beta_k a_{0, k} + \frac{4\beta_k \alpha\Delta t}{h^2}+2\beta_k B\Delta t
\right)
\widetilde{\phi}^{n+1}_{i,j}=\sum_{l=1}^k \left(
-\beta_k a_{l,k} - \beta_k (-1)^l \binom{k}{l} B\Delta t
\right) \widetilde{\phi}^{n+1-l}_{i,j}\\
+\frac{\beta_k \alpha \Delta t }{h^2}\left(
\widetilde{\phi}^{n+1}_{i-1,j}+\widetilde{\phi}^{n+1}_{i+1,j}
+\widetilde{\phi}^{n+1}_{i,j-1}+\widetilde{\phi}^{n+1}_{i,j+1}
\right)
+\beta_k \Delta t \mathcal{N}(\widetilde{\phi}^{n+1}_{i,j}).
\end{multline}

Moreover, the auxiliary iterate generated by the fixed-point update admits the unified representation:
\begin{multline}
\widetilde{\phi}_{i,j}^{n+1, (m+1)}
=\frac{1}{\beta_k a_{0, k} + \frac{4\beta_k \alpha\Delta t}{h^2}+2\beta_k B\Delta t}
\sum_{l=1}^k \left(
-\beta_k a_{l,k} - \beta_k (-1)^l \binom{k}{l} B\Delta t
\right) \phi^{n+1-l}_{i,j} \\
+\frac{\beta_k \alpha \Delta t }{h^2}\left(
\phi^{n+1,(m)}_{i-1,j}+\phi^{n+1,(m)}_{i+1,j}
+\phi^{n+1,(m)}_{i,j-1}+\phi^{n+1,(m)}_{i,j+1}
\right)
+\beta_k \Delta t \mathcal{N}(\phi^{n+1,(m)}_{i,j}).
\end{multline}

\section{Contractivity of the fixed-point iteration}\label{sec:convergence of iteration}

In this section, we establish unconditional contractivity of the fixed-point
iteration associated with the stabilized BDF schemes introduced in Section~2.
The analysis is carried out in a unified framework and applies to the sBDF$k$
schemes with $k=1,2,3,4$.

\begin{theorem}[Unconditional contractivity of the fixed-point iterations]
\label{thm:unified_contractivity}
Let $\{\phi^{n+1,(m)}\}_{m\ge0}$ be the algorithmic iterates generated by the sBDF$k$
schemes in Section~2, where $k\in\{1,2,3,4\}$.
Let $\widetilde{\phi}^{n+1,(m)}$ denote the corresponding auxiliary iterates produced by the
fixed-point update (before applying the cut-off when $k\ge2$).

Assume that the nonlinear term $\mathcal N(\cdot)$ is Lipschitz continuous on $[-1,1]$:
\begin{equation}\label{eq:N_Lipschitz}
|\mathcal N(u)-\mathcal N(v)|\le 2B\,|u-v|,\qquad \forall\,u,v\in[-1,1].
\end{equation}

Then the fixed-point iterations are contractive in the discrete $\ell^2$-norm:
\begin{equation}\label{eq:contractive_unified}
\|\phi^{n+1,(m+1)}-\phi^{n+1,(m)}\|
\le \rho_k\,\|\phi^{n+1,(m)}-\phi^{n+1,(m-1)}\|,
\qquad m\ge1,
\end{equation}
where the contraction factor satisfies $0<\rho_k<1$ for any $\Delta t>0$ and $h>0$, with
\begin{equation}\label{eq:rho_k}
\rho_k=\frac{A_k}{\beta_k a_{0,k}+A_k},\qquad
A_k:=\frac{4\beta_k\alpha\Delta t}{h^2} + 2\beta_k B\,\Delta t.
\end{equation}
Here $(a_{0,k},\beta_k)$ are given by $(a_{0,1},\beta_1)=(1,1)$, $(a_{0,2},\beta_2)=(3/2,2)$, 
$(a_{0,3},\beta_3)=(11/6,6)$, $(a_{0,4},\beta_4)=(25/12,12)$.
In particular, the fixed-point iterations converge unconditionally for all $\Delta t>0$ and $h>0$.
\end{theorem}

\begin{proof}
We write the fixed-point update for sBDF$k$ in the unified form
\begin{equation}\label{eq:FP_unified}
\widetilde{\phi}^{n+1,(m+1)}
=
\frac{1}{\beta_k a_{0,k}+A_k}
\Big(
\mathcal H_k (\phi^n, \phi^{n-1},\dots, \phi^{n-k+1})
+\mathcal L_k \phi^{n+1,(m)}
+\beta_k\Delta t\,\mathcal N(\phi^{n+1,(m)})
\Big),
\end{equation}
where $\mathcal H_k$ collects all history terms independent of the iteration index $m$,
and $\mathcal L_k$ denotes the neighbor-sum part of the diffusion operator $\Delta_h$.
By construction, the coefficient in front of the neighbor sum equals
$\frac{4\beta_k\alpha\Delta t}{h^2}$.

Subtracting the $m$-th update from the $(m+1)$-th update in \eqref{eq:FP_unified},
the history terms cancel since $\mathcal H_k$ is independent of $m$, yielding
\begin{equation}\label{eq:diff_unified}
\widetilde{\phi}^{n+1,(m+1)}-\widetilde{\phi}^{n+1,(m)}
=
\frac{1}{\beta_k a_{0,k}+A_k}
\Big(
\mathcal L_k(\phi^{n+1,(m)}-\phi^{n+1,(m-1)})
+\beta_k\Delta t\big[\mathcal N(\phi^{n+1,(m)})-\mathcal N(\phi^{n+1,(m-1)})\big]
\Big).
\end{equation}
Taking the discrete $\ell^2$-norm on both sides, using the triangle inequality, and
noting that $\mathcal L_k$ consists of a weighted sum of the four nearest neighbors,
we obtain
\[
\|\widetilde{\phi}^{n+1,(m+1)}-\widetilde{\phi}^{n+1,(m)}\|
\le
\frac{\frac{4\beta_k\alpha\Delta t}{h^2}}{\beta_k a_{0,k}+A_k}
\|\phi^{n+1,(m)}-\phi^{n+1,(m-1)}\|
+
\frac{2\beta_k B\Delta t}{\beta_k a_{0,k}+A_k}
\|\phi^{n+1,(m)}-\phi^{n+1,(m-1)}\|.
\]
Therefore,
\begin{equation}\label{eq:tilde_contr}
\|\widetilde{\phi}^{n+1,(m+1)}-\widetilde{\phi}^{n+1,(m)}\|
\le
\frac{A_k}{\beta_k a_{0,k}+A_k}
\|\phi^{n+1,(m)}-\phi^{n+1,(m-1)}\|
= \rho_k\,\|\phi^{n+1,(m)}-\phi^{n+1,(m-1)}\|.
\end{equation}

For $k=1$, we have $\phi^{n+1,(m)}=\widetilde{\phi}^{n+1,(m)}$, so \eqref{eq:contractive_unified}
follows directly from \eqref{eq:tilde_contr}.
For $k\ge2$, the algorithmic iterates are obtained by the pointwise cut-off
$\phi^{n+1,(m)}=\Pi_{[-1,1]}(\widetilde{\phi}^{n+1,(m)})$.
Since $\Pi_{[-1,1]}$ is non-expansive, we have
\[
\|\phi^{n+1,(m+1)}-\phi^{n+1,(m)}\|
\le
\|\widetilde{\phi}^{n+1,(m+1)}-\widetilde{\phi}^{n+1,(m)}\|.
\]
Combining this with \eqref{eq:tilde_contr} yields \eqref{eq:contractive_unified}.
Finally, since $a_{0,k}>0$ and $A_k>0$ for any $\Delta t>0$ and $h>0$, we have
$0<\rho_k=\frac{A_k}{a_{0,k}+A_k}<1$, proving unconditional contractivity.
\end{proof}

\begin{remark}
The stabilization term plays a crucial role in establishing unconditional
contractivity of the fixed-point iteration.
Without stabilization, the contraction factor would depend on $\Delta t$ and
would generally require a restrictive time-step condition.
The inclusion of the stabilization term enlarges the diagonal dominance of the
iteration operator and yields a contraction factor bounded away from one,
uniformly in $\Delta t$ and $h$.
\end{remark}

\begin{remark}
For the sBDF1 scheme, the fixed-point mapping is monotone and preserves the
maximum bound principle automatically; hence no cut-off operator is required.
For higher-order schemes ($k\ge2$), although the fixed-point iteration remains
contractive due to stabilization, the maximum bound principle is not preserved
by the underlying BDF discretization.
A pointwise cut-off operator is therefore incorporated at the algorithmic level
to enforce unconditional MBP, without affecting the contractivity or the
convergence of the iteration.
\end{remark}

\section{Error analysis}\label{sec:error_analysis}

In this section, we establish the convergence of the proposed sBDF$k$ algorithms
for $k=1,2,3,4$.
The analysis is carried out in a unified framework and is based on the unconditional
contractivity of the fixed-point iteration proved in
Theorem~\ref{thm:unified_contractivity}.

Throughout this section, $\phi(\cdot,t)$ denotes the exact solution of
\eqref{eq:semi-discrete}.
Let $\phi^n$ be the numerical solution produced by the sBDF$k$ algorithm at time level
$t_n$.
We further denote by $\widetilde{\phi}^{\,n+1}$ the exact solution of the implicit
stabilized BDF$k$ scheme at time $t_{n+1}$, i.e., the fixed point of the contractive
mapping associated with the iteration.

Define the errors
\[
e^{n+1} := \phi^{n+1} - \phi(\cdot,t_{n+1}),
\qquad
\widetilde e^{\,n+1} := \widetilde{\phi}^{\,n+1}-\phi(\cdot,t_{n+1}).
\]


We assume that the exact solution satisfies the maximum bound
$\phi(\cdot,t)\in[-1,1]$ for $t\in[0,T]$.


\begin{lemma}[Local truncation error]\label{lem:LTE}
Assume that $\phi \in C^{k+1}\big([0,T];C^{4}(\Omega)\big)$ ($k=1,2,3,4$).
Let $\tau^{n+1}$ be the residual obtained by inserting the exact solution
$\phi(\cdot,t_{n+1}),\phi(\cdot,t_n),\dots$ into the implicit stabilized BDF$k$
scheme. Then there exists a constant $C>0$, independent of $\Delta t$ and $h$, such that
\[
\|\tau^{n+1}\| \le C(\Delta t^{\,k}+h^2).
\]
\end{lemma}

\begin{proof}
We outline the consistency estimate in both time and space.

By Taylor expansion of $\phi(t_{n+1-\ell})$ around $t_{n+1}$, we have
\[
\phi(t_{n+1-\ell})
= \sum_{r=0}^{k}\frac{(-\ell\Delta t)^r}{r!}\,\partial_t^r\phi(t_{n+1})
+ O(\Delta t^{k+1}),
\]
and using the order conditions of BDF$k$ yields
\[
D_k\phi(t_{n+1}) = \partial_t\phi(t_{n+1}) + O(\Delta t^{k}).
\]
The stabilization term in the sBDF$k$ scheme is a linear combination of backward
differences up to order $k$, hence it is also $O(\Delta t^{k})$ when evaluated at the
exact solution. Therefore the total temporal truncation error is $O(\Delta t^k)$.

Next, we consider the spatial discretization error. 
Let $\Delta_h$ be the five-point Laplacian. For $\phi(\cdot,t_{n+1})\in C^{4}(\Omega)$,
a standard Taylor expansion gives
\[
(\Delta_h\phi)_{ij} = (\Delta\phi)(x_{ij},t_{n+1}) + O(h^2),
\]
uniformly on interior grid points. The same estimate holds in the discrete $\ell^2$-norm.

Putting together, and noting that the remaining lower-order terms are
evaluated consistently at $t_{n+1}$, we obtain
\[
\|\tau^{n+1}\|\le C(\Delta t^{k}+h^2),
\]
where $C$ depends on $\|\phi\|_{C^{k+1}([0,T];C^4(\Omega))}$ but is independent of
$\Delta t$ and $h$.
\end{proof}

\begin{lemma}[Error recursion via an energy argument]\label{lem:error_recursion}
Let $\widetilde{\phi}^{\,n+1}$ be the exact solution of the implicit stabilized
BDF$k$ scheme. Then there exists a constant $C>0$, independent of $\Delta t$ and $h$, such that
\begin{equation}\label{eq:error_recursion_energy}
\|\widetilde e^{\,n+1}\|
\le
(1+C\Delta t)\sum_{j=0}^{k-1}\|e^{n-j}\|
+ C\Delta t(\Delta t^{\,k}+h^2),
\qquad k=1,2,3,4 .
\end{equation}
\end{lemma}

\begin{proof}
We start from the unified sBDF$k$ formulation.
Using the same notation and coefficients as in Section~3, the implicit
sBDF$k$ scheme can be written in the form
\begin{equation}\label{eq:sBDFk_energy_form}
\beta_k a_{0,k}\,\widetilde{\phi}^{\,n+1}
- \beta_k\alpha\Delta t\,\Delta_h \widetilde{\phi}^{\,n+1}
+ 2\beta_k B\Delta t\,\widetilde{\phi}^{\,n+1}
=
\mathcal H_k (\widetilde{\phi})
+ \beta_k\Delta t\,\mathcal N(\widetilde{\phi}^{\,n+1}),
\end{equation}
where $\mathcal H_k(\widetilde{\phi})$ collects all history terms involving
$\widetilde{\phi}^n,\widetilde{\phi}^{n-1},\dots,\widetilde{\phi}^{n-k+1}$.

Substituting the exact solution $\phi(\cdot,t_{n+1})$ into
\eqref{eq:sBDFk_energy_form} yields a residual $\tau^{n+1}$ satisfying
$\|\tau^{n+1}\|\le C(\Delta t^{\,k}+h^2)$ (Lemma~\ref{lem:LTE}).
Subtracting the resulting equation from \eqref{eq:sBDFk_energy_form}, we obtain the
error equation
\begin{multline}\label{eq:error_eq_energy}
\beta_k a_{0,k}\,\widetilde e^{\,n+1}
- \beta_k\alpha\Delta t\,\Delta_h \widetilde e^{\,n+1}
+ 2\beta_k B\Delta t\,\widetilde e^{\,n+1}
=
(\mathcal H_k (\widetilde{\phi}) - \mathcal H_k(\phi) )\\
+ \beta_k\Delta t
\bigl(\mathcal N(\widetilde{\phi}^{\,n+1})
- \mathcal N(\phi(\cdot,t_{n+1}))\bigr)
- \Delta t\,\tau^{n+1}.
\end{multline}

Taking the discrete $\ell^2$ inner product of \eqref{eq:error_eq_energy} with
$\widetilde e^{\,n+1}$ yields
\begin{multline}
\beta_k a_{0,k}\|\widetilde e^{\,n+1}\|^2
\;-\; \beta_k\alpha\Delta t\,
(\Delta_h \widetilde e^{\,n+1},\widetilde e^{\,n+1})
+ 2 \beta_k B\Delta t\|\widetilde e^{\,n+1}\|^2  
=
(\mathcal H_k (\widetilde{\phi}) 
 - \mathcal H_k(\phi),\widetilde e^{\,n+1})\\
+ \beta_k\Delta t\,
(\mathcal N(\widetilde{\phi}^{\,n+1})
- \mathcal N(\phi(\cdot,t_{n+1})),\widetilde e^{\,n+1})
- \Delta t(\tau^{n+1},\widetilde e^{\,n+1}).
\end{multline}

By the discrete Green's identity, we obtain
\[
\beta_k a_{0,k}\|\widetilde e^{\,n+1}\|^2
+ \beta_k\alpha\Delta t\,\|\nabla_h \widetilde e^{\,n+1}\|^2
+ 2 \beta_k B\Delta t\|\widetilde e^{\,n+1}\|^2
\ge \beta_k a_{0,k}\|\widetilde e^{\,n+1}\|^2.
\]

Next, we estimate the right-hand side.

For the History terms, the difference $\mathcal H_k(\widetilde{\phi})-\mathcal H_k(\phi)$ consists of linear combinations of
$e^n$, $\dots$, $e^{n-k+1}$ with coefficients depending only on $k$ and $B\Delta t$.
Therefore,
\[
|(\mathcal H_k(\widetilde{\phi})-\mathcal H_k(\phi),\widetilde e^{\,n+1})|
\le
C(1+\Delta t)\sum_{j=0}^{k-1}\|e^{n-j}\|\,
\|\widetilde e^{\,n+1}\|.
\]

For the nonlinear term, using the Lipschitz continuity \eqref{eq:N_props},
\[
(\mathcal N(\widetilde{\phi}^{\,n+1})
- \mathcal N(\phi(\cdot,t_{n+1})),\widetilde e^{\,n+1})
\le
2B\|\widetilde e^{\,n+1}\|^2.
\]
This term is absorbed by the stabilization contribution
$\beta_k B\Delta t\|\widetilde e^{\,n+1}\|^2$ on the left-hand side.

For the truncation error, we apply the Cauchy--Schwarz and Lemma~\ref{lem:LTE},
\[
|(\tau^{n+1},\widetilde e^{\,n+1})|
\le \|\tau^{n+1}\|\,\|\widetilde e^{\,n+1}\|
\le C(\Delta t^{\,k}+h^2)\|\widetilde e^{\,n+1}\|.
\]

Combining the above estimates yields
\[
a_{0,k}\|\widetilde e^{\,n+1}\|^2
\le
C(1+\Delta t)\sum_{j=0}^{k-1}\|e^{n-j}\|\,
\|\widetilde e^{\,n+1}\|
+ C\Delta t(\Delta t^{\,k}+h^2)\|\widetilde e^{\,n+1}\|.
\]
Dividing by $a_{0,k}\|\widetilde e^{\,n+1}\|$ gives
\eqref{eq:error_recursion_energy}.
\end{proof}

\begin{lemma}[Termination error of the fixed-point iteration]
\label{lem:termination_error}
Let $\widetilde{\phi}^{\,n+1}$ be the fixed point of the iteration and
$\phi^{n+1}$ the algorithmic solution obtained by terminating the iteration once
\[
\|\phi^{n+1,(m+1)}-\phi^{n+1,(m)}\| < \varepsilon .
\]
Then
\[
\|\phi^{n+1}-\widetilde{\phi}^{\,n+1}\|
\le \frac{\varepsilon}{1-\rho_k}.
\]
\end{lemma}

\begin{proof}
Let $m_\star$ be the stopping index and set $\phi^{n+1}:=\phi^{n+1,(m_\star)}$.
Since the iteration is contractive with factor $\rho_k\in(0,1)$, we have for any $r\ge0$
\[
\|\phi^{n+1,(m_\star+r+1)}-\phi^{n+1,(m_\star+r)}\|
\le \rho_k^{\,r}\,
\|\phi^{n+1,(m_\star+1)}-\phi^{n+1,(m_\star)}\|.
\]
By the stopping criterion,
$\|\phi^{n+1,(m_\star+1)}-\phi^{n+1,(m_\star)}\|<\varepsilon$.
Moreover, the fixed point $\widetilde{\phi}^{\,n+1}$ satisfies
\[
\phi^{n+1,(m_\star)}-\widetilde{\phi}^{\,n+1}
= \sum_{r=0}^{\infty}\bigl(\phi^{n+1,(m_\star+r)}-\phi^{n+1,(m_\star+r+1)}\bigr),
\]
hence by the triangle inequality,
\[
\|\phi^{n+1}-\widetilde{\phi}^{\,n+1}\|
\le
\sum_{r=0}^{\infty}
\|\phi^{n+1,(m_\star+r+1)}-\phi^{n+1,(m_\star+r)}\|
\le
\sum_{r=0}^{\infty}\rho_k^{\,r}\,\varepsilon
=
\frac{\varepsilon}{1-\rho_k}.
\]
This proves the lemma.
\end{proof}


\begin{theorem}[Error estimate]\label{thm:global_error}
Let the assumptions of Lemma~\ref{lem:LTE} hold, and choose the stopping tolerance as
\[
\varepsilon = C_\varepsilon \min\{\Delta t^{\,k+1},h^2\}.
\]
Then there exists a constant $C>0$, independent of $\Delta t$ and $h$, such that
\[
\|e^{n}\|
\le
C e^{CT}\Big(\sum_{j=0}^{k-1}\|e^{j}\| + \Delta t^{\,k}+h^2\Big),
\qquad t_n\le T .
\]
In particular, if the starting values satisfy
$\sum_{j=0}^{k-1}\|e^{j}\|\le C(\Delta t^{\,k}+h^2)$, then
\[
\|e^{n}\|\le C e^{CT}(\Delta t^{\,k}+h^2),
\qquad t_n\le T .
\]
\end{theorem}

\begin{proof}
We decompose the total error at time $t_{n+1}$ as
\[
e^{n+1}
=
\phi^{n+1}-\phi(\cdot,t_{n+1})
=
\underbrace{\bigl(\phi^{n+1}-\widetilde{\phi}^{\,n+1}\bigr)}_{\text{termination error}}
+
\underbrace{\bigl(\widetilde{\phi}^{\,n+1}-\phi(\cdot,t_{n+1})\bigr)}_{\widetilde e^{\,n+1}}.
\]
For $k=1$, no cut-off is applied. For $k\ge2$, the pointwise cut-off operator is
non-expansive and satisfies $\Pi_{[-1,1]}(\phi(\cdot,t_{n+1}))=\phi(\cdot,t_{n+1})$, hence
the same estimate applies to the algorithmic solution $\phi^{n+1}$.

By Lemma~\ref{lem:error_recursion} and Lemma~\ref{lem:termination_error},
\[
\|e^{n+1}\|
\le
(1+C\Delta t)\sum_{j=0}^{k-1}\|e^{n-j}\|
+ C\Delta t(\Delta t^{\,k}+h^2)
+ \frac{\varepsilon}{1-\rho_k}.
\]
With the choice $\varepsilon=C_\varepsilon\min\{\Delta t^{k+1},h^2\}$ and $\Delta t\le 1$,
we have $\frac{\varepsilon}{1-\rho_k}\le C(\Delta t^{k+1}+h^2)\le C(\Delta t^k+h^2)$, hence
\begin{equation}\label{eq:one_step_rec_detail}
\|e^{n+1}\|
\le
(1+C\Delta t)\sum_{j=0}^{k-1}\|e^{n-j}\|
+ C\Delta t(\Delta t^{\,k}+h^2).
\end{equation}

Define $E^n:=\sum_{j=0}^{k-1}\|e^{n-j}\|$ for $n\ge k-1$.
Summing \eqref{eq:one_step_rec_detail} over the $k$ consecutive indices
$n,n-1,\dots,n-k+1$ gives
\[
E^{n+1}\le (1+C\Delta t)E^n + C\Delta t(\Delta t^{\,k}+h^2).
\]
Iterating this inequality yields
\[
E^n \le (1+C\Delta t)^{n-(k-1)}E^{k-1}
+ C\Delta t(\Delta t^{\,k}+h^2)\sum_{m=k-1}^{n-1}(1+C\Delta t)^{n-1-m}.
\]
Using $(1+C\Delta t)^{n}\le e^{CT}$ for $t_n\le T$ and a geometric-series bound for the sum,
we obtain
\[
E^n \le e^{CT}E^{k-1} + Ce^{CT}(\Delta t^{\,k}+h^2).
\]
Since $\|e^n\|\le E^n$, this proves the desired global estimate. The final statement follows
from the assumed accuracy of the starting values.
\end{proof}

\section{Discrete maximum bound principle}\label{sec:MBP}

In this section, we establish unconditional preservation of the discrete maximum
bound principle (MBP) for the proposed sBDF schemes.
Throughout, we assume that the continuous problem \eqref{eq:semi-discrete}
satisfies the maximum bound principle
\[
|\phi(x,t)| \le 1, \qquad \forall (x,t)\in \Omega\times[0,T],
\]
and that the nonlinear term $\mathcal N(\cdot)$ satisfies
\begin{equation}\label{eq:N_bounded}
|\mathcal N(s)| \le B, \qquad \forall s\in[-1,1].
\end{equation}

\begin{theorem}[Discrete MBP for sBDF1]\label{thm:MBP_sBDF1}
For any time-step size $\Delta t>0$, the sBDF1 algorithm preserves the discrete
maximum bound principle, i.e.,
\begin{equation}\label{eq:MBP_sBDF1}
\|\phi^{n}\|_{\infty} \le 1, \qquad \forall n\ge0,
\end{equation}
provided that the initial data satisfy $\|\phi^{0}\|_{\infty}\le1$.
\end{theorem}

\begin{proof}
We consider the fixed-point iteration associated with sBDF1.
Let $\widetilde{\phi}^{n+1,(m)}$ denote the $m$-th iterate, with the initialization
$\widetilde{\phi}^{n+1,(0)}=\phi^{n}$.

Assume inductively that
\[
|\phi^{n}_{i,j}| \le 1,
\qquad
|\widetilde{\phi}^{n+1,(m)}_{i,j}| \le 1
\quad \text{for all } (i,j).
\]
Using the explicit fixed-point update for sBDF1, we obtain
\[
\begin{aligned}
|\widetilde{\phi}^{n+1,(m+1)}_{i,j}|
&\le
\frac{(1+B\Delta t)|\phi^{n}_{i,j}|}{1+\frac{4\alpha\Delta t}{h^2}+2B\Delta t}
+ \frac{\alpha\Delta t}{h^2}
\frac{\sum_{\text{nbr}} |\widetilde{\phi}^{n+1,(m)}|}{1+\frac{4\alpha\Delta t}{h^2}+2B\Delta t} \\
&\quad
+ \frac{\Delta t\,|\mathcal N(\widetilde{\phi}^{n+1,(m)}_{ij})|}{1+\frac{4\alpha\Delta t}{h^2}+2B\Delta t}.
\end{aligned}
\]
Invoking the induction hypothesis and the boundedness assumption
\eqref{eq:N_bounded}, we obtain
\[
|\widetilde{\phi}^{n+1,(m+1)}_{i,j}|
\le
\frac{1+B\Delta t+\frac{4\alpha\Delta t}{h^2}+B\Delta t}
     {1+\frac{4\alpha\Delta t}{h^2}+2B\Delta t}
=1.
\]
Therefore, the bound $|\widetilde{\phi}^{n+1,(m)}_{i,j}|\le1$ holds for all $m$.
Taking the limit as $m\to m_\star$ yields
$|\phi^{n+1}_{i,j}|\le1$, completing the proof.
\end{proof}

\begin{theorem}[Discrete MBP for sBDF$k$, $k=2,3,4$]\label{thm:MBP_sBDFk}
Let $k\in\{2,3,4\}$. For any $\Delta t>0$, the sBDF$k$ algorithms preserve the discrete
maximum bound principle:
\begin{equation}\label{eq:MBP_sBDFk}
\|\phi^{n}\|_{\infty} \le 1, \qquad \forall n\ge0,
\end{equation}
provided that
\[
\|\phi^{n-\ell}\|_{\infty}\le1, \qquad \ell=0,1,\dots,k-1.
\]
\end{theorem}

\begin{proof}
At time level $t_{n+1}$, let $\widetilde{\phi}^{\,n+1,(m+1)}$ denote the auxiliary
iterate generated by the fixed-point update before applying the cut-off.
The algorithmic iterate is then defined by
\[
\phi^{n+1,(m+1)} = \Pi_{[-1,1]}\!\left(\widetilde{\phi}^{\,n+1,(m+1)}\right),
\]
where $\Pi_{[-1,1]}$ denotes the pointwise projection onto $[-1,1]$.

By definition of the projection operator, we have
\[
\phi^{n+1,(m+1)}_{i,j} \in [-1,1],
\qquad \forall (i,j), \ \forall m\ge0.
\]
Therefore,
\[
\|\phi^{n+1,(m)}\|_{\infty} \le 1
\quad \text{for all } m,
\]
and in particular,
\[
\|\phi^{n+1}\|_{\infty}
= \|\phi^{n+1,(m_\star)}\|_{\infty}
\le 1.
\]
This completes the proof.
\end{proof}

\section{Discrete energy stability analysis}\label{sec:energy}

In this section, we establish unconditional discrete energy stability of the 
proposed sBDF schemes applied to the Allen--Cahn equation.
The Allen--Cahn equation serves as a prototypical example of a gradient-flow
system with a nonconvex potential, and therefore provides a stringent test for
energy-stable time discretizations.

The Allen--Cahn equation reads
\begin{equation}
\partial_t \phi = \alpha \Delta \phi + f(\phi),
\qquad
f(\phi)=\phi-\phi^3,
\label{eq:AC}
\end{equation}
which can be interpreted as the $L^2$-gradient flow associated with the
Ginzburg--Landau energy functional
\begin{equation}
E(\phi)
=
\int_\Omega
\left(
\frac{\alpha}{2} |\nabla \phi|^2
+ F(\phi)
\right)\,dx,
\qquad
F(\phi)=\frac14(\phi^2-1)^2.
\label{eq:AC_energy}
\end{equation}
The continuous solution satisfies the energy dissipation law
\begin{equation}
\frac{d}{dt}E(\phi(t))
=
-\left\|
\frac{\delta E}{\delta \phi}
\right\|_{2}^2
\le 0.
\label{eq:AC_energy_decay}
\end{equation}

Let $u$ be a grid function defined on a uniform Cartesian mesh with mesh
size $h$. We introduce the following discrete energy functional for Allen-Cahn equation.
\begin{definition}[Discrete Energy]
\begin{equation}
E_{h}[u]=
\frac{\alpha}{2}\|\nabla_h u \|^2 + (F(u), 1).
\label{eq:discrete_energy}
\end{equation}
\end{definition}

We first recall the discrete Green's identity, which will be used in
the energy analysis.

\begin{lemma}[Discrete Green's formula]
\label{lem:discrete_green}
Let $u$ and $v$ be grid functions satisfying periodic, homogeneous Neumann, or
homogeneous Dirichlet boundary conditions. Then
\begin{equation}
( \nabla_h u , \nabla_h v )
=
-( u , \Delta_h v ) .
\end{equation}
\end{lemma}

For $k\ge 2$, the algorithm applies the pointwise cut-off $\Pi_{[-1,1]}$ after each FPI update 
to enforce the maximum bound principle .
The next lemma shows that for Allen--Cahn equation, this cut-off does not increase the discrete energy.

\begin{lemma}[Energy non-increase under cut-off]\label{lem:cutoff_AC}
Let $\Pi_{[-1,1]}(s):=\max\{\min\{s,1\},-1\}$ be applied pointwise to grid functions.
Then for any grid function $u$,
\begin{equation}\label{eq:cutoff_energy_AC}
E_h\big(\Pi_{[-1,1]}(u)\big)\le E_h(u).
\end{equation}
\end{lemma}
\begin{proof}
The projection operator is non-expansive, i.e.,
\[
|\Pi_{[-1,1]}(a)-\Pi_{[-1,1]}(b)|\le |a-b|,
\qquad \forall a,b\in\mathbb{R}.
\]
Applying this inequality to all discrete gradients yields
$\|\nabla_h \Pi_{[-1,1]}(u)\| \le \|\nabla_h u\|$.
Moreover, since $F(s)=\frac14(s^2-1)^2$ is monotone outside $[-1,1]$, the pointwise
projection does not increase the potential energy.
Combining the two parts proves the result.
\end{proof}

We will also need the following two technical lemmas for the energy stability analysis.

\begin{lemma}
Let $\tilde{\phi}^{n+1,(m+1)}$ be the auxiliary iterate produced by the fixed-point update of the sBDF$k$ scheme ($k \in \{1,2,3,4\}$) before applying the cut-off operator. Let $\delta^{m+1} := \tilde{\phi}^{n+1,(m+1)} - \phi^{n+1,(m)}$. The iterative difference satisfies the following discrete inner product identity:
\begin{align*}
\frac{\tilde{D}_{k}}{\beta_{k}\Delta t}\|\delta^{m+1}\|_{2}^{2}=\frac{1}{\beta_{k}\Delta t}\left[\mathcal{H}_{k}(\phi^n, \dots,\phi^{n-k+1})
-(a_{0,k}+\beta_{k}B\Delta t)\phi^{n+1,(m)},\delta^{m+1}\right]\\
+\alpha\left(\Delta_{h}\phi^{n+1,(m)},\delta^{m+1}\right)+\left( f(\phi^{n+1,(m)}),\delta^{m+1}\right),
\end{align*}
where $\tilde{D}_k := \beta_k a_{0,k} + \frac{4\beta_k \alpha \Delta t}{h^2} + 2\beta_k B \Delta t$, and $(a_{0,k}, \beta_k)$ are the scheme-dependent coefficients defined previously.
\end{lemma}

\begin{proof}
By the unified formulation of the fixed-point update, we have:
\begin{equation}
\tilde{\phi}^{n+1,(m+1)}=\frac{1}{\beta_k a_{0,k}+A_{k}}\left(\mathcal{H}_{k}+\mathcal{L}_{k}\phi^{n+1,(m)}+\beta_{k}\Delta t\mathcal{N}(\phi^{n+1,(m)})\right).
\end{equation}
We have $A_{k}=\frac{4\beta_{k}\alpha\Delta t}{h^{2}}+2\beta_{k}B\Delta t$. The diffusion term $\mathcal{L}_k$ is defined via the neighbor-sum operator.
Substituting $\mathcal{N}(\phi)=f(\phi)+B\phi$ and moving all terms involving $\phi^{n+1,(m)}$ to the right hand side, we rewrite the update equation in terms of $\delta^{m+1}$:
\begin{align*}
\left(\beta_k a_{0,k}+A_k\right)\delta^{m+1}=\mathcal{H}_{k}-a_{0,k}\phi^{n+1,(m)}+\beta_{k}\alpha\Delta t\Delta_{h}\phi^{n+1,(m)}
-\beta_{k}B\Delta t\phi^{n+1,(m)}+\beta_{k}\Delta t f(\phi^{n+1,(m)}).
\end{align*}
Taking the discrete $l^2$ inner product of both sides with $\delta^{m+1}$ and dividing by $\beta_k \Delta t$ directly yields the stated identity.
\end{proof}

\begin{lemma}[Auxiliary Quadratic Functional]
Define the unified auxiliary quadratic functional as:
\begin{equation}
\mathbb{H}_{k}(\phi):=\frac{ a_{0,k}+B\Delta t}{2\Delta t}\left\|\phi-\frac{\mathcal{H}_{k}}{\beta_k a_{0,k}+\beta_{k}B\Delta t}\right\|_{2}^{2}.
\end{equation}
The difference of this functional between two successive iterations evaluated at $\tilde{\phi}^{n+1,(m+1)}$ and $\phi^{n+1,(m)}$ is exactly:
\begin{equation}
\mathbb{H}_{k}(\tilde{\phi}^{n+1,(m+1)})-\mathbb{H}_{k}(\phi^{n+1,(m)})=\frac{1}{\beta_k \Delta t}\langle(\beta_k a_{0,k}+\beta_k B\Delta t)\phi^{n+1,(m)}-\mathcal{H}_{k},\delta^{m+1}\rangle+\frac{a_{0,k}+B\Delta t}{2\Delta t}\|\delta^{m+1}\|_{2}^{2}.
\end{equation}
\end{lemma}

\begin{proof}
Let $C = \frac{a_{0,k}+B\Delta t}{2\Delta t}$ and $V = \frac{\mathcal{H}_k}{\beta_k a_{0,k}+\beta_k B\Delta t}$. Expanding the difference $C\|\tilde{\phi}^{n+1,(m+1)} - V\|_2^2 - C\|\phi^{n+1,(m)} - V\|_2^2$ using the algebraic identity $\|a\|_2^2 - \|b\|_2^2 = \langle a-b, a+b \rangle$ with $a-b = \delta^{m+1}$ and $a+b = 2(\phi^{n+1,(m)} - V) + \delta^{m+1}$, we obtain:
\begin{equation}
C\langle\delta^{m+1}, 2(\phi^{n+1,(m)} - V) + \delta^{m+1}\rangle = 2C\langle\phi^{n+1,(m)} - V, \delta^{m+1}\rangle + C\|\delta^{m+1}\|_2^2.
\end{equation}
Substituting $C$ and $V$ back into the equation completes the proof.
\end{proof}

\begin{theorem}[Unconditional Discrete Energy Stability]
Let $\phi^{n+1,(m)}$ be the $m$-th algorithmic iterate produced by the sBDF$k$ scheme ($k \in \{1,2,3,4\}$). Define the generalized augmented energy as $\mathcal{E}_{k}^{m}:=E_{h}[\phi^{n+1,(m)}]+\mathbb{H}_{k}(\phi^{n+1,(m)})$. For any time step $\Delta t > 0$, $\mathcal{E}_{k}^{m}$ is unconditionally non-increasing with respect to the iteration index $m$, i.e., $\mathcal{E}_{k}^{m+1} \le \mathcal{E}_{k}^{m}$. Upon convergence, the scheme satisfies the global energy decay law 
$$
E_{h}[\phi^{n+1}]+\mathbb{H}_{k}(\phi^{n+1})\le E_{h}[\phi^{n}]+\mathbb{H}_{k}(\phi^{n}).
$$
\end{theorem}

\begin{proof}
For $k \ge 2$, the algorithmic iterate is obtained via the projection operator $\phi^{n+1,(m+1)} = \Pi_{[-1,1]}(\tilde{\phi}^{n+1,(m+1)})$. By Lemma 6.2, the pointwise cut-off  ensures $E_{h}[\phi^{n+1,(m+1)}]\le E_{h}[\tilde{\phi}^{n+1,(m+1)}]$, and similarly $\mathbb{H}_{k}(\phi^{n+1,(m+1)})\le \mathbb{H}_{k}(\tilde{\phi}^{n+1,(m+1)})$. Thus, bounding the energy difference for the pre-cut-off update $\tilde{\phi}$ provides a sufficient upper bound.

Applying the Taylor expansion on the nonlinear potential and discrete Green's identity, the difference in the discrete energy is bounded by:
\begin{align*}
E_{h}[\tilde{\phi}^{n+1,(m+1)}]-E_{h}[\phi^{n+1,(m)}]\le & -\langle f(\phi^{n+1,(m)}),\delta^{m+1}\rangle-\alpha\langle\delta^{m+1},\Delta_{h}\phi^{n+1,(m)}\rangle\\
&+B\|\delta^{m+1}\|_{2}^{2}+\frac{\alpha}{2}\|\nabla_{h}\delta^{m+1}\|_{2}^{2}.
\end{align*}

Substituting the result of Lemma 6.3 to eliminate the terms involving $f(\phi)$ and $\Delta_h \phi$, we obtain:
\begin{align*}
E_{h}[\tilde{\phi}^{n+1,(m+1)}]-E_{h}[\phi^{n+1,(m)}]\le & - \left(\frac{\tilde{D}_{k}}{\beta_{k}\Delta t}-B\right)\|\delta^{m+1}\|_{2}^{2}+\frac{\alpha}{2}\|\nabla_{h}\delta^{m+1}\|_{2}^{2}\\&
+\frac{1}{\beta_{k}\Delta t}\left(\mathcal{H}_{k}-(\beta_k a_{0,k}
+\beta_{k}B\Delta t)\phi^{n+1,(m)},\delta^{m+1}\right).
\end{align*}

By the standard property of finite differences, $\frac{\alpha}{2}\|\nabla_{h}\delta^{m+1}\|_{2}^{2}\le\frac{4\alpha}{h^{2}}\|\delta^{m+1}\|_{2}^{2}$. Combining the coefficients for $\|\delta^{m+1}\|_2^2$ simplifies the expression: $\frac{\tilde{D}_{k}}{\beta_{k}\Delta t}-B-\frac{4\alpha}{h^{2}} = \frac{a_{0,k}}{\Delta t}+B$. Therefore:
\begin{equation}
E_{h}[\tilde{\phi}^{n+1,(m+1)}]-E_{h}[\phi^{n+1,(m)}]\le-\left(\frac{a_{0,k}}{\Delta t}+B\right)\|\delta^{m+1}\|_{2}^{2}+\frac{1}{\beta_{k}\Delta t}\langle\mathcal{H}_{k}-(\beta_k a_{0,k}+\beta_{k}B\Delta t)\phi^{n+1,(m)},\delta^{m+1}\rangle.
\end{equation}

Adding the auxiliary functional difference from Lemma 6.4, the linear inner product cross-terms perfectly cancel out, yielding:
\begin{equation}
\mathcal{E}_{k}^{m+1}-\mathcal{E}_{k}^{m}\le-\left(\frac{a_{0,k}}{\Delta t}+B-\frac{a_{0,k}+B\Delta t}{2\Delta t}\right)\|\delta^{m+1}\|_{2}^{2}=-\left(\frac{a_{0,k}}{2\Delta t}+\frac{B}{2}\right)\|\delta^{m+1}\|_{2}^{2}.
\end{equation}

Since $B > 0$, $\Delta t > 0$, and the sBDF coefficient $a_{0,k} > 0$ for all $k \in \{1,2,3,4\}$, we have $\frac{a_{0,k}}{2\Delta t}+\frac{B}{2} > 0$. Hence, $\mathcal{E}_{k}^{m+1} \le \mathcal{E}_{k}^{m}$ unconditionally. Letting $m \rightarrow \infty$, the global discrete energy dissipation law holds.
\end{proof}


\section{Numerical experiments} \label{sec:numerical experiments}

In this section, we present two numerical examples to assess the performance of the proposed
schemes. The first example considers the Allen--Cahn equation posed on $\Omega=[0,1]^2$ with
mixed boundary conditions. This test is used to verify the temporal
convergence orders of the proposed methods and to demonstrate their ability to preserve the MBP
and maintain unconditional energy stability. We also compare the computational efficiency of the
proposed sBDF schemes with first- and second-order ETD schemes. To further examine robustness
in a more challenging setting, we consider a second example based on a prostate cancer growth
model.

\medskip
\noindent\textbf{Example 7.1 (Allen--Cahn equation).}\label{ex:ac}
We consider
\begin{equation}\label{eq:ac_case}
\frac{\partial u}{\partial t}=\varepsilon^{2}\Delta u + u - u^{3},
\qquad \boldsymbol{x}\in\Omega,\ \ t\in(0,T],
\end{equation}
with the initial condition
\begin{equation}\label{eq:7.2}
u_0(x,y)=0.05\big(1-\cos(2\pi x)\big)\cos(2\pi y),
\qquad (x,y)\in\overline{\Omega},
\end{equation}
and mixed boundary conditions consisting of a homogeneous Dirichlet condition on the left edge
of $\Omega$ and homogeneous Neumann conditions on the remaining three edges.

We set $\varepsilon=0.01$ and choose $B=2$ to satisfy the stability requirement $B\ge \|f'\|$.
Table~~\ref{tab:FPI1234-temporal} reports the numerical results at $T=1$, where the $L^{2}$-errors are computed against a
reference solution obtained with a sufficiently small time step to ensure high temporal accuracy.
Tables~\ref{tab:first_order_compare}--\ref{tab:second_order_compare} summarize the numerical errors, observed convergence rates, and computational costs
of the first- and second-order ETD and sBDF schemes. Furthermore, the MBP preservation of the
proposed sBDF schemes is examined over the long-time interval $t\in[0,60]$; see Fig.~\ref{fig:mbp_preservation}.

\begin{table}[htbp]
\centering
\setlength{\tabcolsep}{5.5pt}  
\renewcommand{\arraystretch}{1.15}  
\caption{$L^2$ temporal errors and observed convergence rates for the sBDF schemes applied to equation \eqref{eq:ac_case}
on a fixed $2048\times 2048$ spatial grid.}
\label{tab:FPI1234-temporal}
\begin{tabular}{c|c|cc|cc|cc|cc}
\toprule
\textbf{} & $\Delta t$ & sBDF1 & Order & sBDF2 & Order & sBDF3 & Order & sBDF4 & Order \\
\midrule
\multirow{5}{*}{\rotatebox{90}{Temporal}} 
& 0.1     & 1.37e-2 & --   & 3.50e-3 & --   & 1.21e-4 & --   & 6.92e-6  & --   \\
& 0.05    & 7.60e-3 & 0.85 & 9.83e-4 & 1.83 & 1.78e-5 & 2.76 & 5.63e-7  & 3.62 \\
& 0.025   & 4.00e-3 & 0.92 & 2.57e-4 & 1.93 & 2.40e-6 & 2.89 & 3.96e-8  & 3.83 \\
& 0.0125  & 2.00e-3 & 1.00 & 6.57e-5 & 1.96 & 3.12e-7 & 2.95 & 2.60e-9  & 3.93 \\
& 0.00625 & 1.00e-3 & 1.00 & 1.66e-5 & 1.98 & 3.97e-8 & 2.97 & 1.57e-10 & 4.05 \\
\bottomrule
\end{tabular}
\end{table}

\begin{figure}[t]
    \centering
    \includegraphics[width=0.36\textwidth]{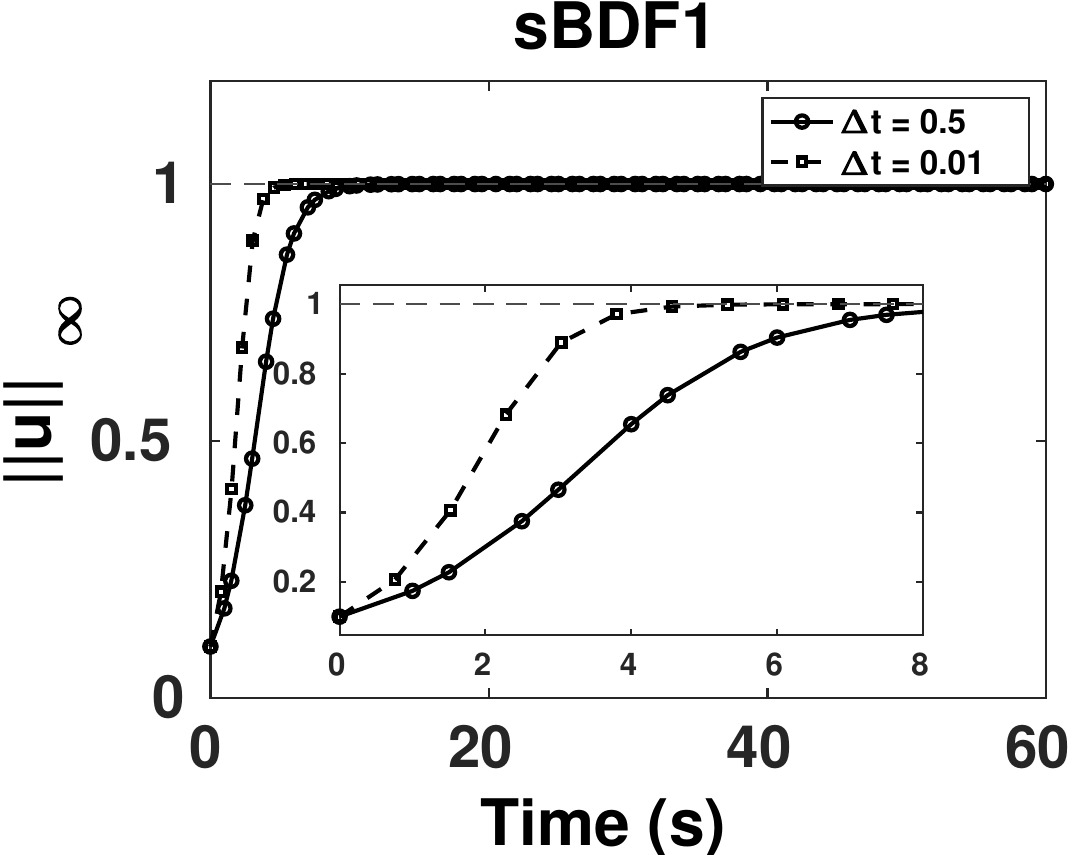}
    \hspace{0.05\textwidth}
    \includegraphics[width=0.36\textwidth]{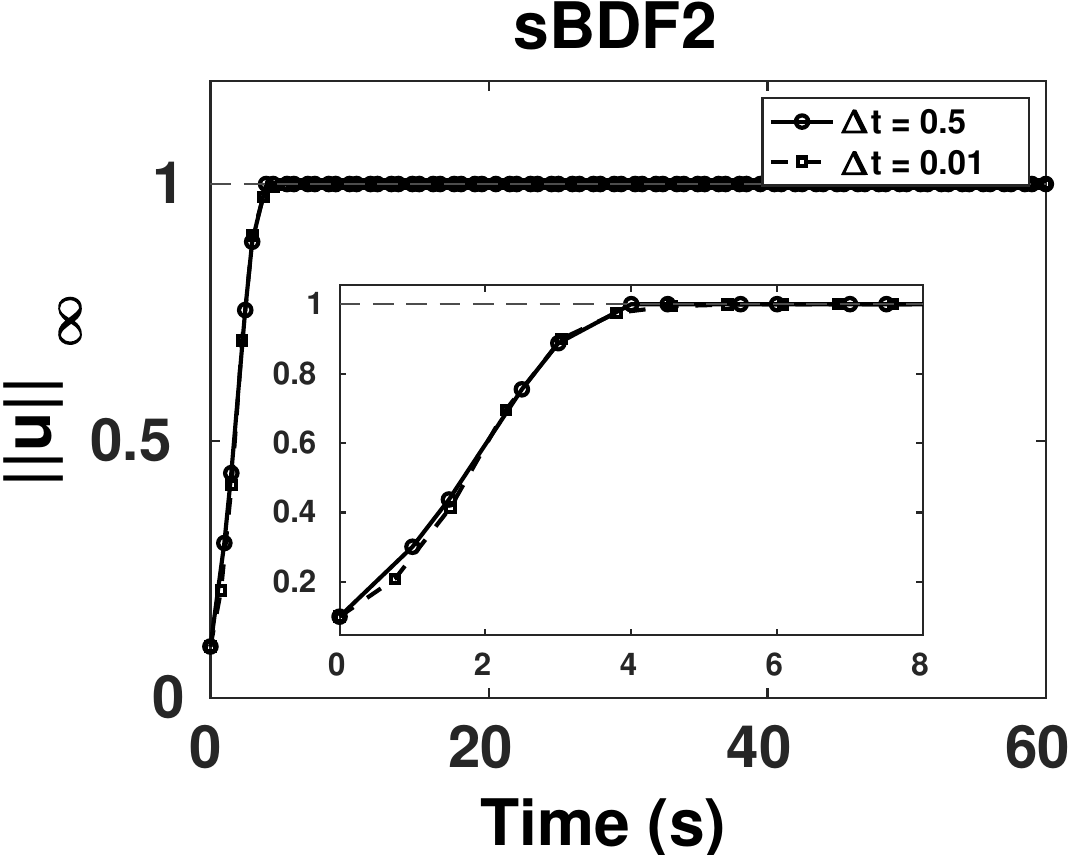}
    
    \vspace{0.02\textwidth}
    
    \includegraphics[width=0.36\textwidth]{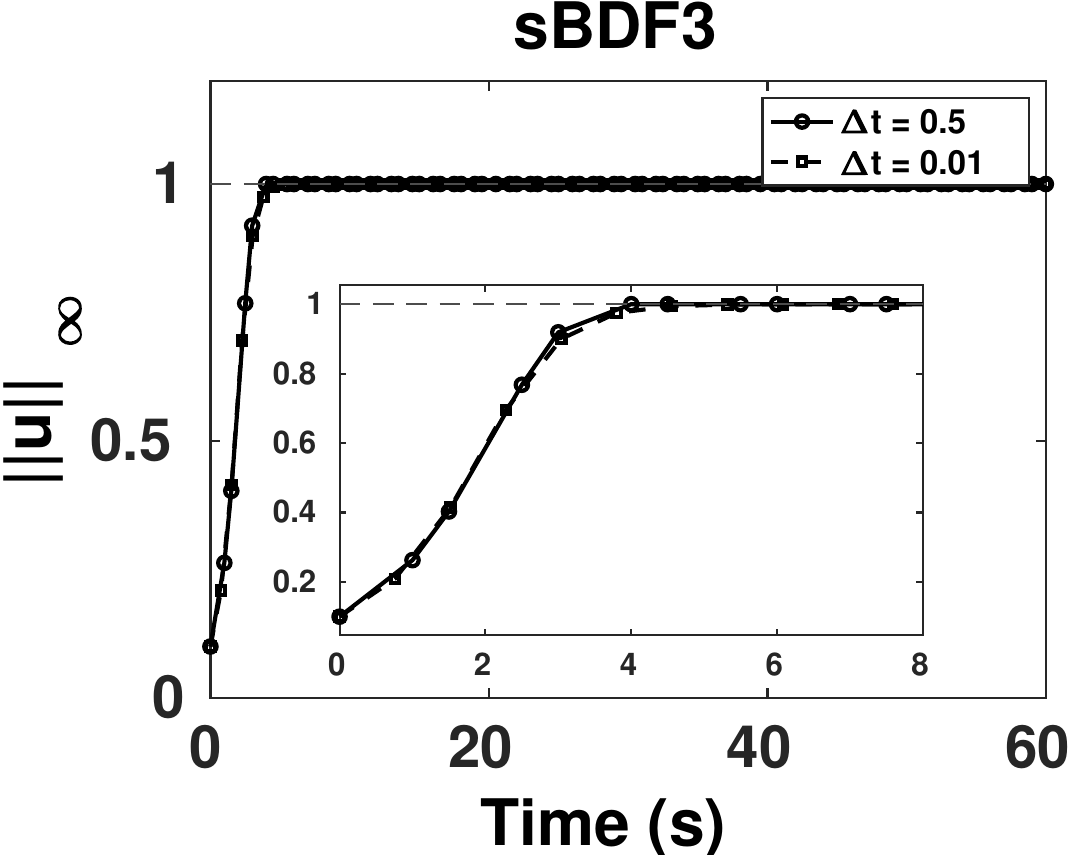}
    \hspace{0.05\textwidth}
    \includegraphics[width=0.36\textwidth]{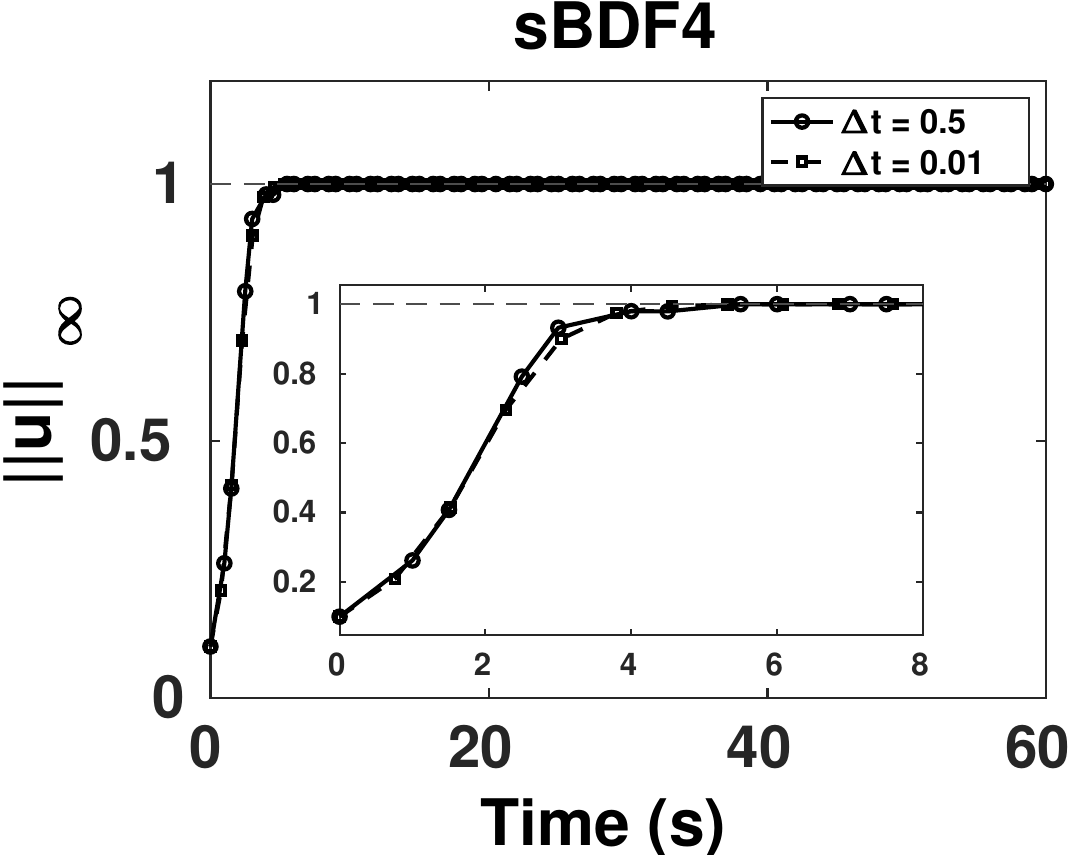}
    \caption{Evolution of the maximum norms of the numerical solutions in Example~7.1 on a fixed $512\times512$ grid.}
    \label{fig:mbp_preservation}
\end{figure}

\begin{figure}[t]
    \centering
    \includegraphics[width=0.36\textwidth]{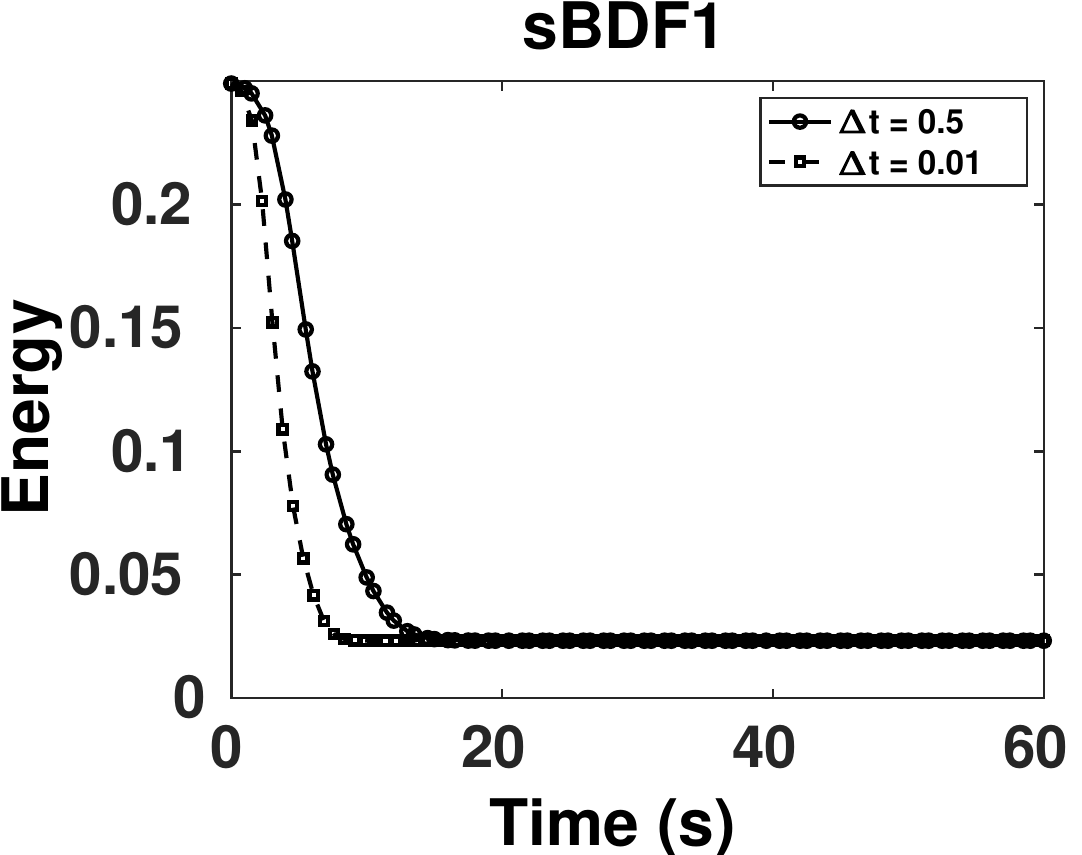}
    \hspace{0.05\textwidth}
    \includegraphics[width=0.36\textwidth]{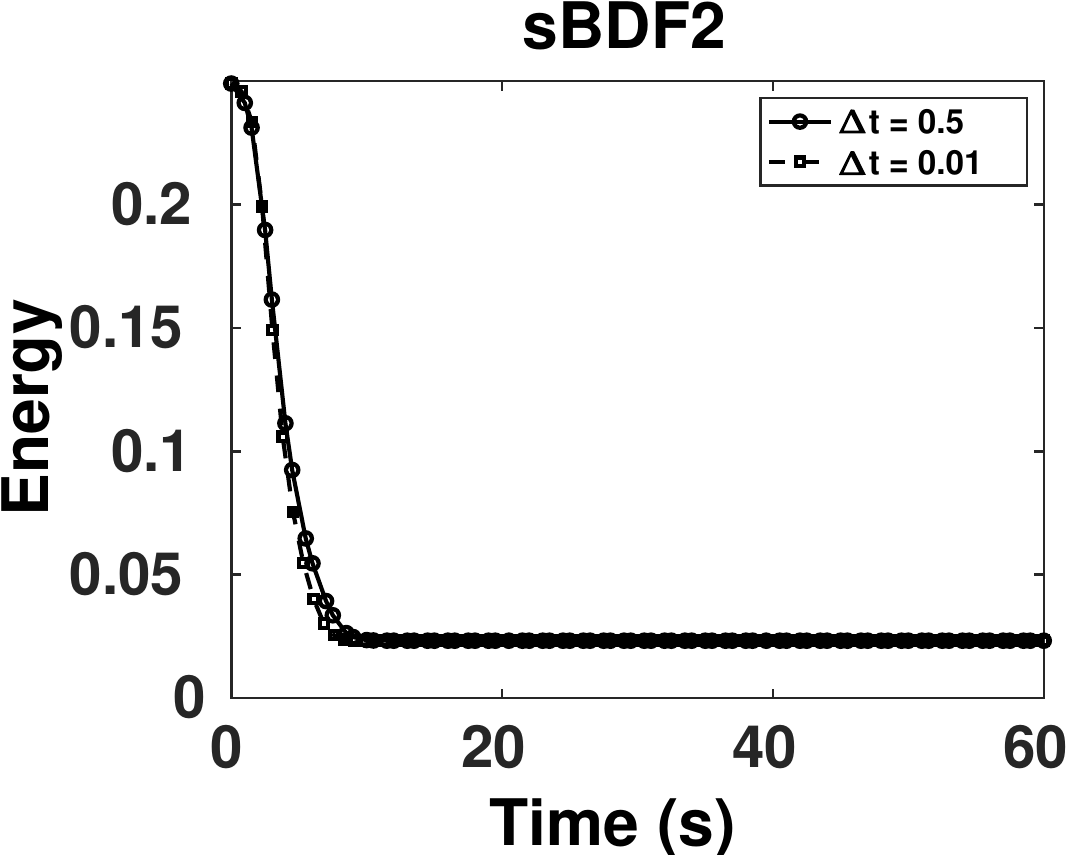}
    
    \vspace{0.02\textwidth}
    
    \includegraphics[width=0.36\textwidth]{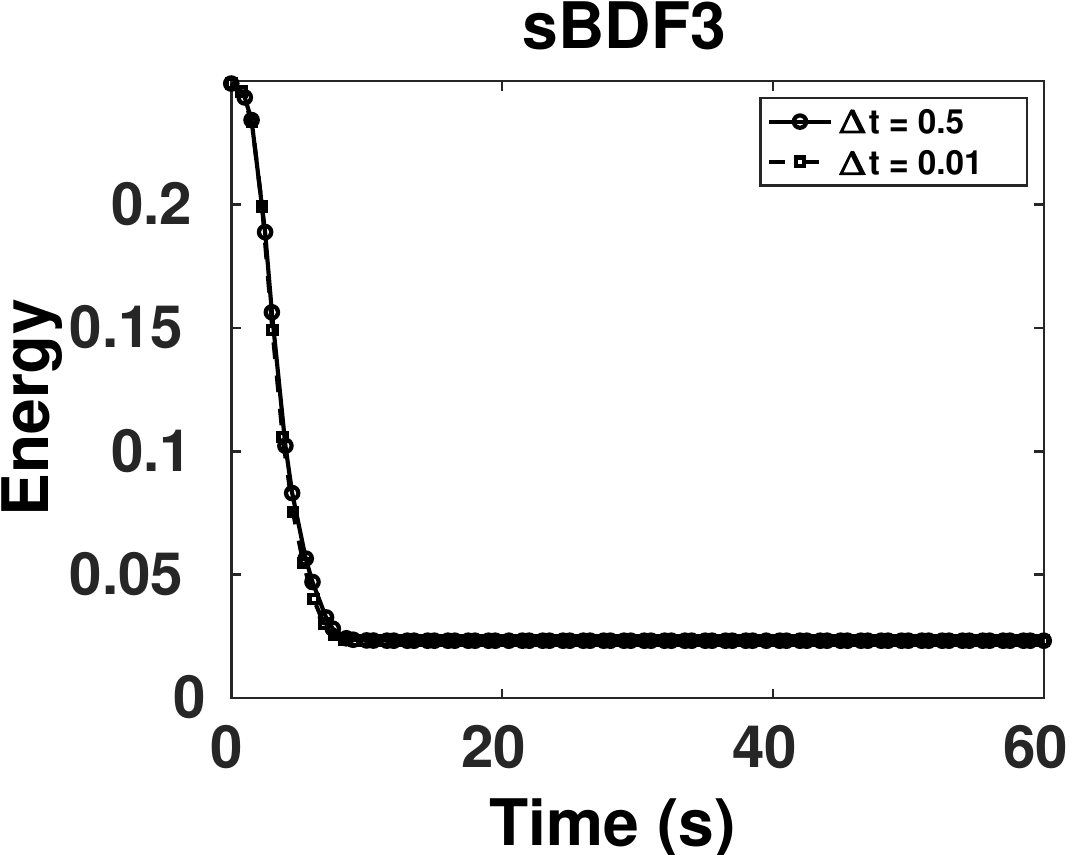}
    \hspace{0.05\textwidth}
    \includegraphics[width=0.36\textwidth]{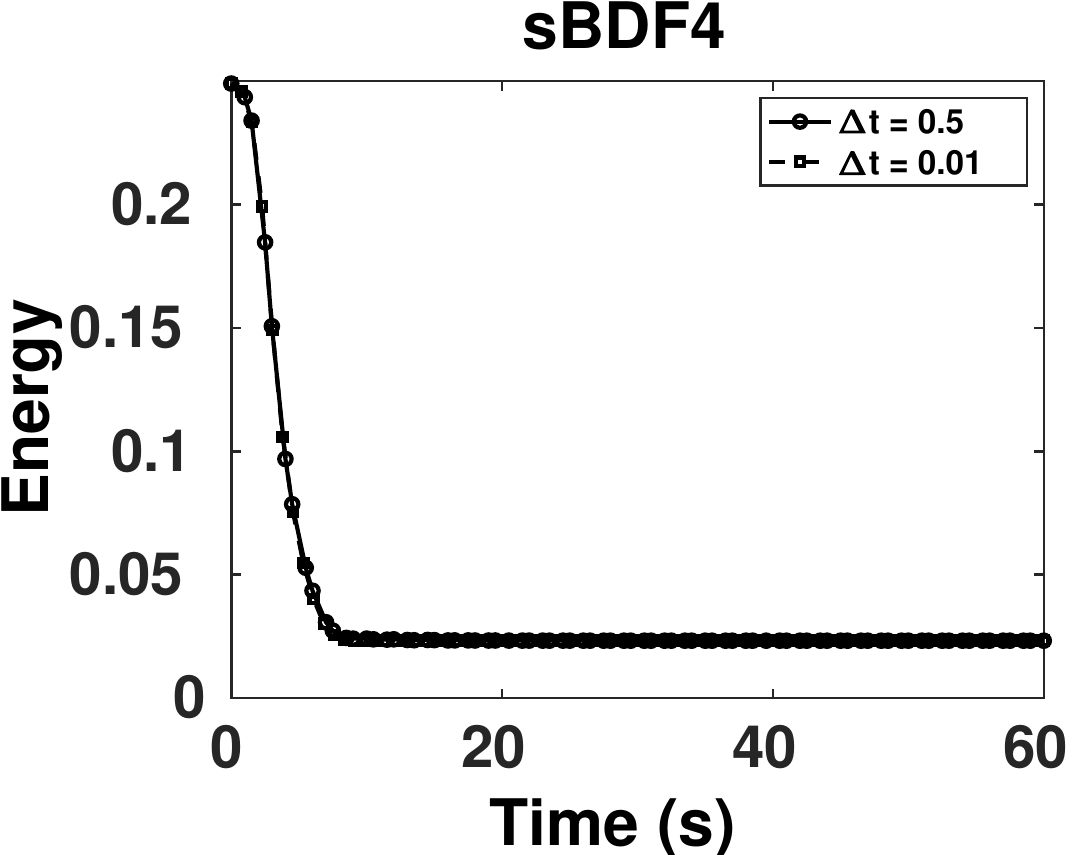}
    \caption{Energy evolution of the numerical solutions in Example~7.1 on a fixed $512\times512$ spatial grid.}
    \label{fig:energy_dissipation}
\end{figure}

\begin{table}[ht]
  \centering
  \small
  \caption{\(L^{2}\) numerical errors, observed convergence rates, and computational times for the ETD1 and sBDF1 schemes applied to equation \eqref{eq:ac_case} on a fixed $512 \times 512$ spatial grid.}
  \label{tab:first_order_compare}
  \begin{tabular}{c
                  S[table-format=1.2e-1] S[table-format=1.2] S[table-format=3.2]
                  S[table-format=1.2e-1] S[table-format=1.2] S[table-format=3.2]}
    \toprule
    {$\Delta t$} 
      & \multicolumn{3}{c}{ETD1} 
      & \multicolumn{3}{c}{sBDF1} \\
    \cmidrule(lr){2-4} \cmidrule(lr){5-7}
      & {Error} & {Order} & {Time [s]} 
      & {Error} & {Order} & {Time [s]} \\
    \midrule
    0.1     & 1.35e-2 &      & 9.26  & 1.36e-2 &      & 6.01 \\
    0.05    & 7.24e-3 & 0.89 & 18.56 & 7.39e-3 & 0.88 & 9.88 \\
    0.025   & 3.71e-3 & 0.97 & 37.06 & 3.81e-3 & 0.95 & 11.07 \\
    0.0125  & 1.82e-3 & 1.03 & 74.16 & 1.88e-3 & 1.02 & 12.99 \\
    0.00625 & 8.36e-4 & 1.12 & 148.51& 8.70e-4 & 1.11 & 16.19 \\
    \bottomrule
  \end{tabular}
\end{table}

\begin{table}[ht]
  \centering
  \small
  \caption{\(L^{2}\) numerical errors, observed convergence rates, and computational times for the ETDRK2 and sBDF2 schemes applied to equation \eqref{eq:ac_case} on a fixed $512 \times 512$ spatial grid.}
  \label{tab:second_order_compare}
  \begin{tabular}{c
                  S[table-format=1.2e-1] S[table-format=1.2] S[table-format=3.2]
                  S[table-format=1.2e-1] S[table-format=1.2] S[table-format=3.2]}
    \toprule
    {$\Delta t$} 
      & \multicolumn{3}{c}{ETDRK2} 
      & \multicolumn{3}{c}{sBDF2} \\
    \cmidrule(lr){2-4} \cmidrule(lr){5-7}
      & {Error} & {Order} & {Time [s]} 
      & {Error} & {Order} & {Time [s]} \\
    \midrule
    0.1     & 1.75e-3 &     & 14.62 & 4.42e-4 &      & 9.30 \\
    0.05    & 4.85e-4 & 1.85& 29.21 & 1.77e-4 & 1.32 & 10.61 \\
    0.025   & 1.28e-4 & 1.93& 57.18 & 5.26e-5 & 1.75 & 11.89 \\
    0.0125  & 3.27e-5 & 1.97&115.69 & 1.42e-5 & 1.89 & 14.20 \\
    0.00625 & 8.21e-6 & 1.99&228.84 & 3.60e-6 & 1.97 & 18.50 \\
    \bottomrule
  \end{tabular}
\end{table}

\medskip
\noindent\textbf{Example 7.2 (Prostate cancer model \cite{Colli2020}).}\label{ex:pca}
We next consider a prostate cancer growth model that describes the spatiotemporal evolution of
the tumor order parameter $\phi$, nutrient concentration $\sigma$, and PSA level $p$ through the
coupled reaction--diffusion system
\begin{align}
\phi_t &= \lambda \Delta \phi - 2M\,\phi(1-\phi)\Big(1-2\phi-3\big(m(\sigma)-m_{\mathrm{ref}}u\big)\Big),
\label{eq:7.3}\\
\sigma_t &= \eta \Delta \sigma + S_h(1-\phi) + (S_c-s)\phi - \big(\gamma_h(1-\phi)+\gamma_c\phi\big)\sigma,
\label{eq:7.4}\\
p_t &= D\Delta p - \gamma_p p + \alpha_h(1-\phi) + \alpha_c \phi.
\label{eq:7.5}
\end{align}
The initial and boundary conditions are given by
\begin{equation}\label{eq:7.6}
\phi(\cdot,0)=\phi_0,\qquad \sigma(\cdot,0)=\sigma_0,\qquad p(\cdot,0)=p_0 \quad \text{in }\Omega;
\qquad
\phi=0,\ \ \partial_{\boldsymbol n}\sigma=\partial_{\boldsymbol n}p=0 \quad \text{on }\partial\Omega.
\end{equation}

The simulations are performed on a square domain of side length $L_d=3000\,\mu\mathrm{m}$,
uniformly discretized with $N=2048$ grid points in each spatial direction. We use the time step
$\Delta t=0.01$ and integrate up to the final time $T=365$ days. All remaining model parameters
are taken from \cite{Colli2020}.

The initial condition for the tumor phase-field variable $\phi$ is prescribed as
\begin{equation}\label{eq:phi0}
\phi_0(x,y)
=
0.5-0.5\,\tanh\!\left(
10\sqrt{\left(\frac{x-L_d/2}{a}\right)^2+\left(\frac{y-L_d/2}{b}\right)^2}-1
\right),
\end{equation}
where $a=150\,\mu\mathrm{m}$ and $b=200\,\mu\mathrm{m}$.
The initial nutrient and PSA concentrations are defined by
\begin{equation}\label{eq:sigma0_p0}
\sigma_0(x,y)=c_\sigma^{0}+c_\sigma^{1}\phi_0(x,y),
\qquad
p_0(x,y)=c_p^{0}+c_p^{1}\phi_0(x,y),
\end{equation}
with
\begin{equation}\label{eq:cp_constants}
c_\sigma^{0}=1~\mathrm{g/L},\quad c_\sigma^{1}=-0.8~\mathrm{g/L},\quad
c_p^{0}=0.0625~\mathrm{ng}\cdot\mathrm{mL}^{-1}\cdot\mathrm{cc}^{-1},\quad
c_p^{1}=0.7975~\mathrm{ng}\cdot\mathrm{mL}^{-1}\cdot\mathrm{cc}^{-1}.
\end{equation}

The simulations investigate both natural tumor growth and tumor responses under therapeutic
protocols. For brevity, we report only the aggressive-tumor case under baseline conditions and
consider three different therapies.

\begin{figure}[H]
  \centering
  \includegraphics[width=0.92\textwidth]{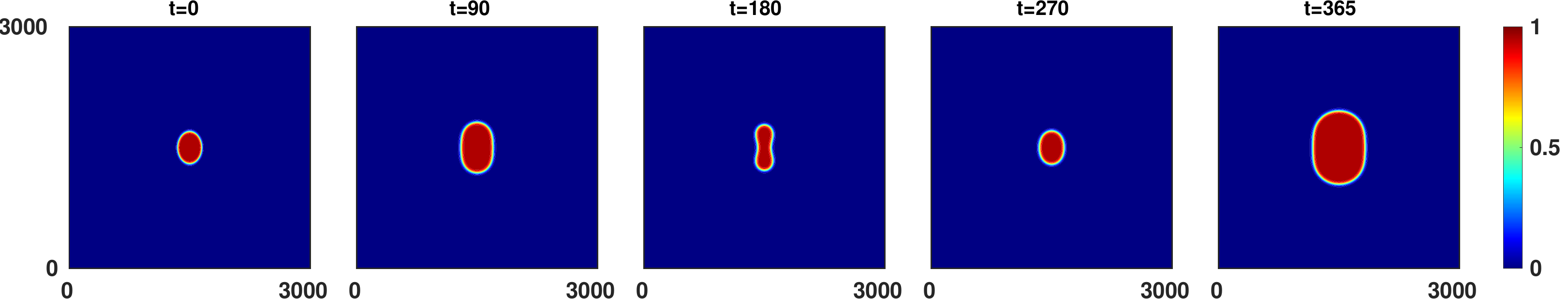}\par\vspace{0.35em}
  \includegraphics[width=0.92\textwidth]{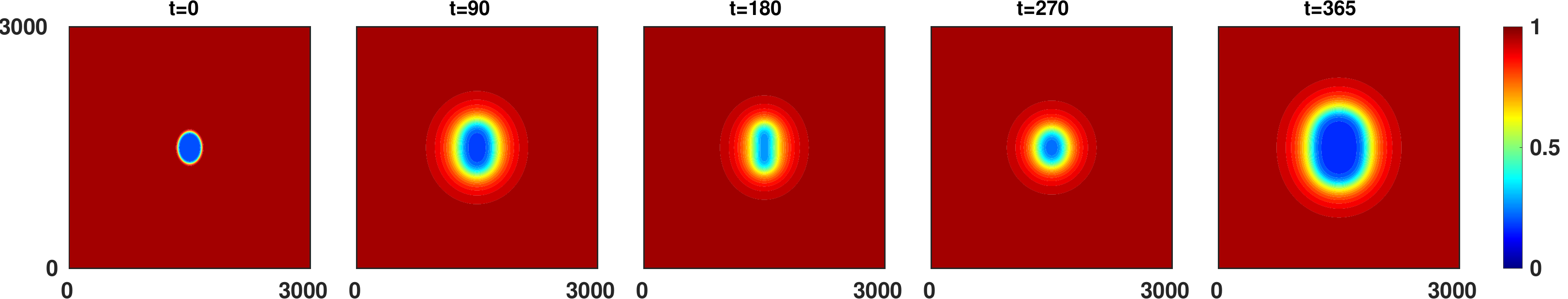}\par\vspace{0.35em}
  \includegraphics[width=0.92\textwidth]{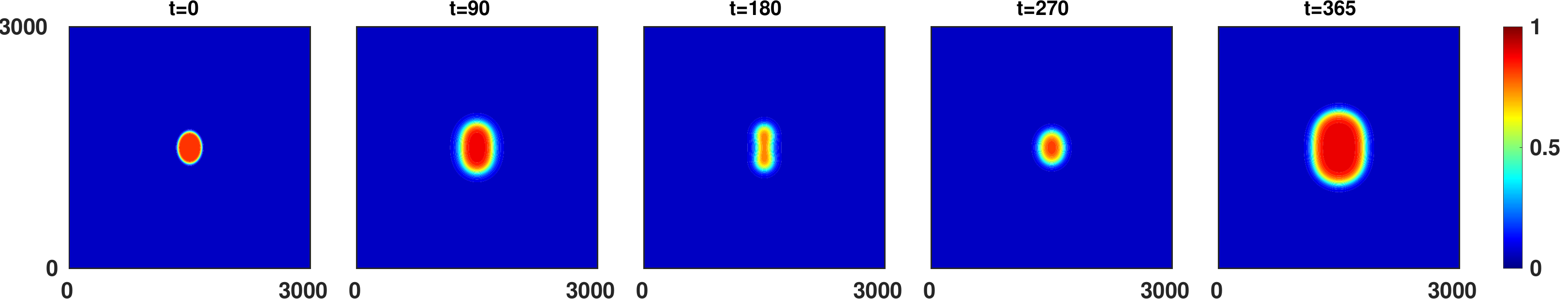}
  \caption{Tumor growth (first row), nutrient distribution (second row), and PSA distribution (third row) for an aggressive tumor under baseline conditions with (I) pure cytotoxic drug therapy, with \(S_c = 2.75\) and \(\gamma_c = 17\).}
  \label{fig:snapshots_three_rows}
\end{figure}

\begin{figure}[H]
  \centering
  \includegraphics[width=0.92\textwidth]{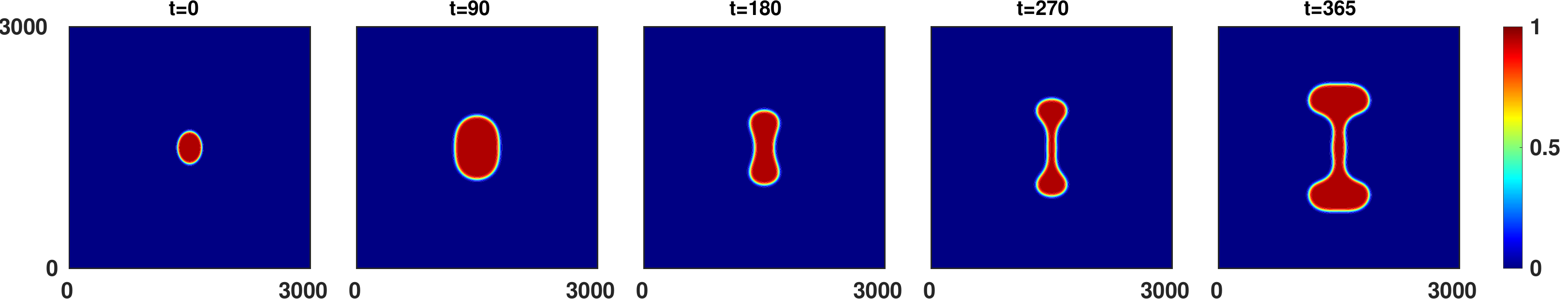}\par\vspace{0.35em}
  \includegraphics[width=0.92\textwidth]{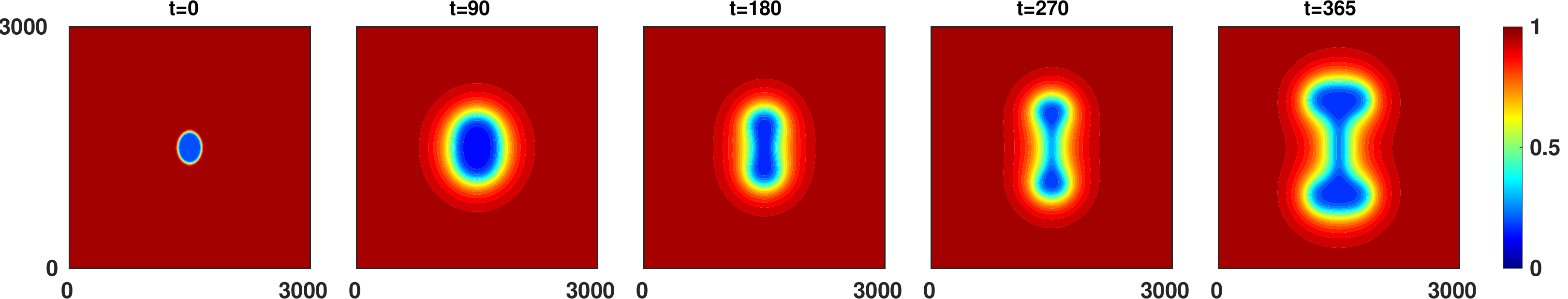}\par\vspace{0.35em}
  \includegraphics[width=0.92\textwidth]{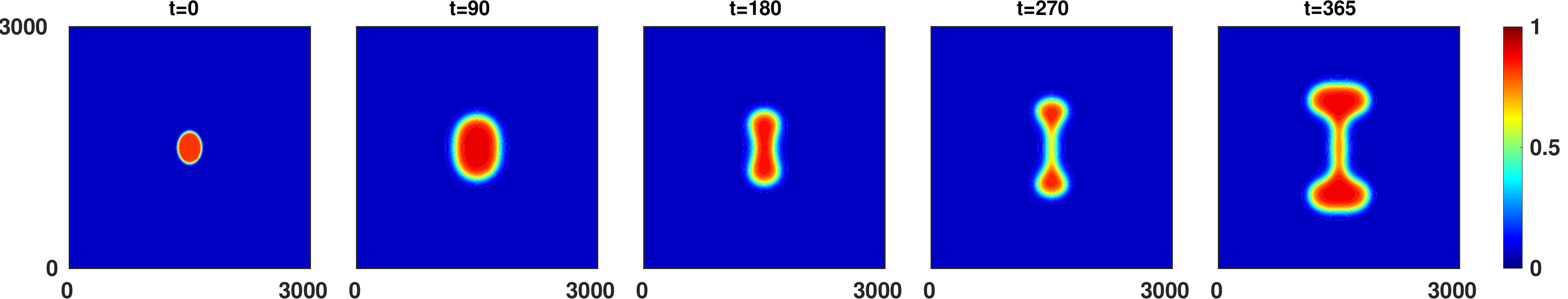}
  \caption{Tumor growth (the first row), Nutrient distribution (the second row) and PSA distribution (the third row) of an aggressive tumor with (II) pure anti-angiogenic drug therapy under baseline, with \(S_c = 2.75\) and \(\gamma_c = 17\).}
  \label{fig:snapshots_three_rows}
\end{figure}

\begin{figure}[H]
  \centering
  \includegraphics[width=0.92\textwidth]{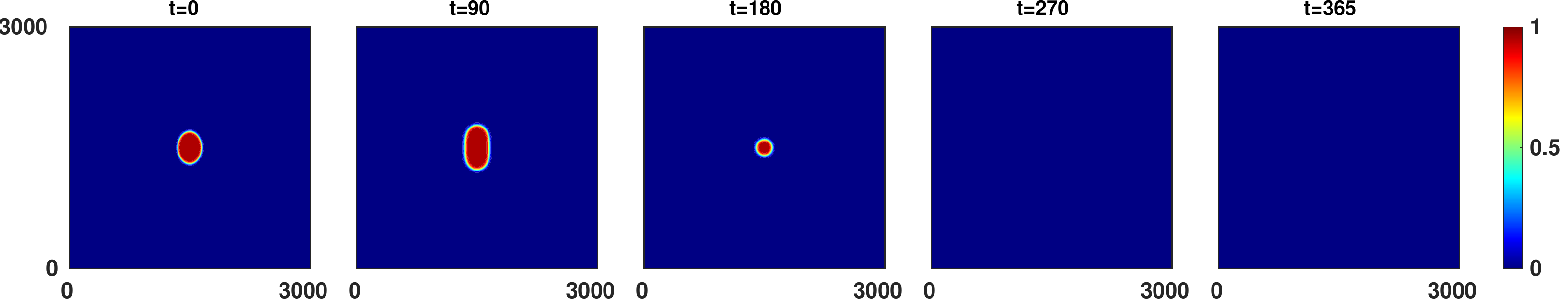}\par\vspace{0.35em}
  \includegraphics[width=0.92\textwidth]{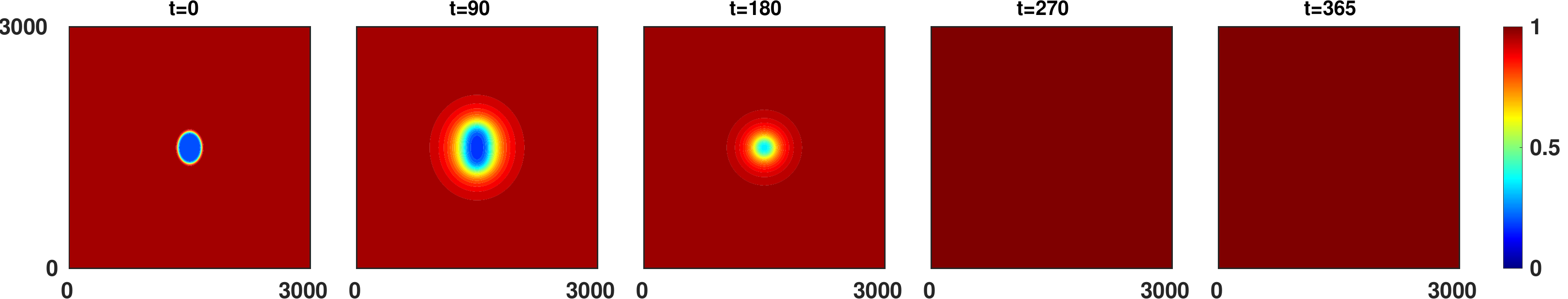}\par\vspace{0.35em}
  \includegraphics[width=0.92\textwidth]{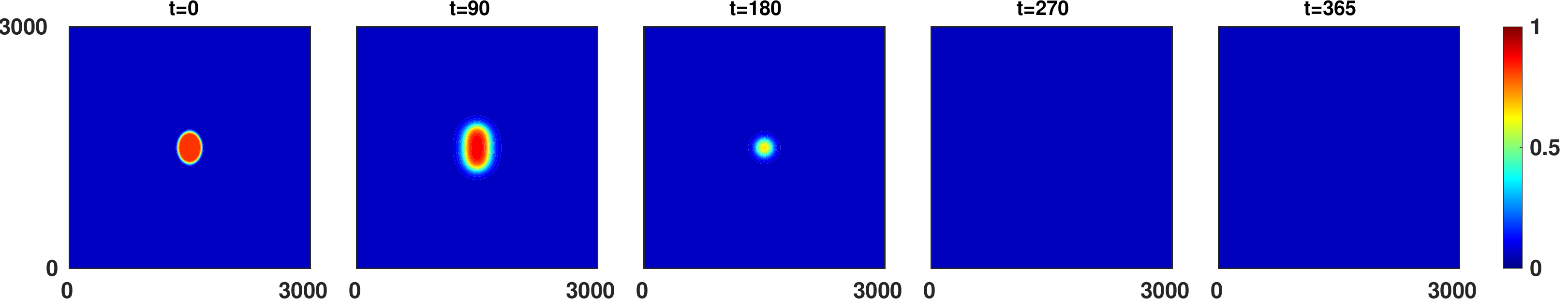}
  \caption{Tumor growth (the first row), Nutrient distribution (the second row) and PSA distribution (the third row) of an aggressive tumor with (III) combined therapy under baseline, with \(S_c = 2.75\) and \(\gamma_c = 17\).}
  \label{fig:snapshots_three_rows}
\end{figure}

\begin{figure}[H]
    \centering
    \includegraphics[width=0.3\textwidth]{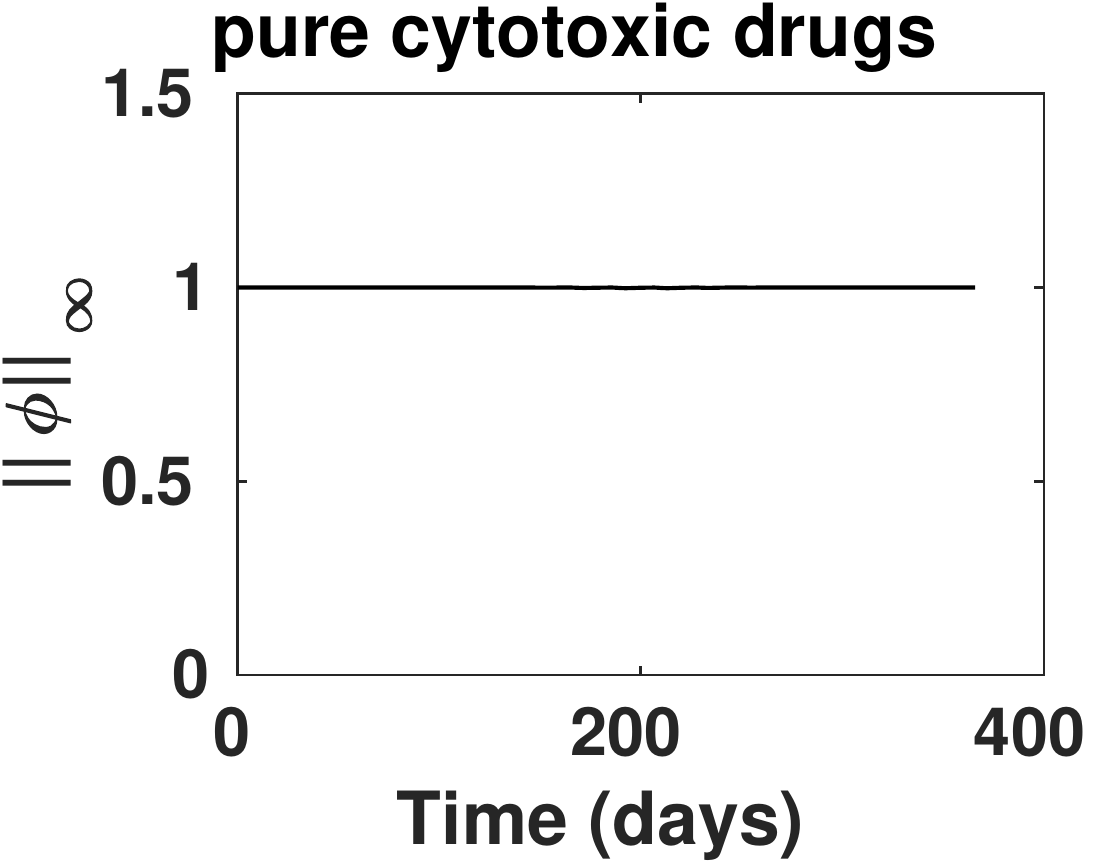}
    \includegraphics[width=0.3\textwidth]{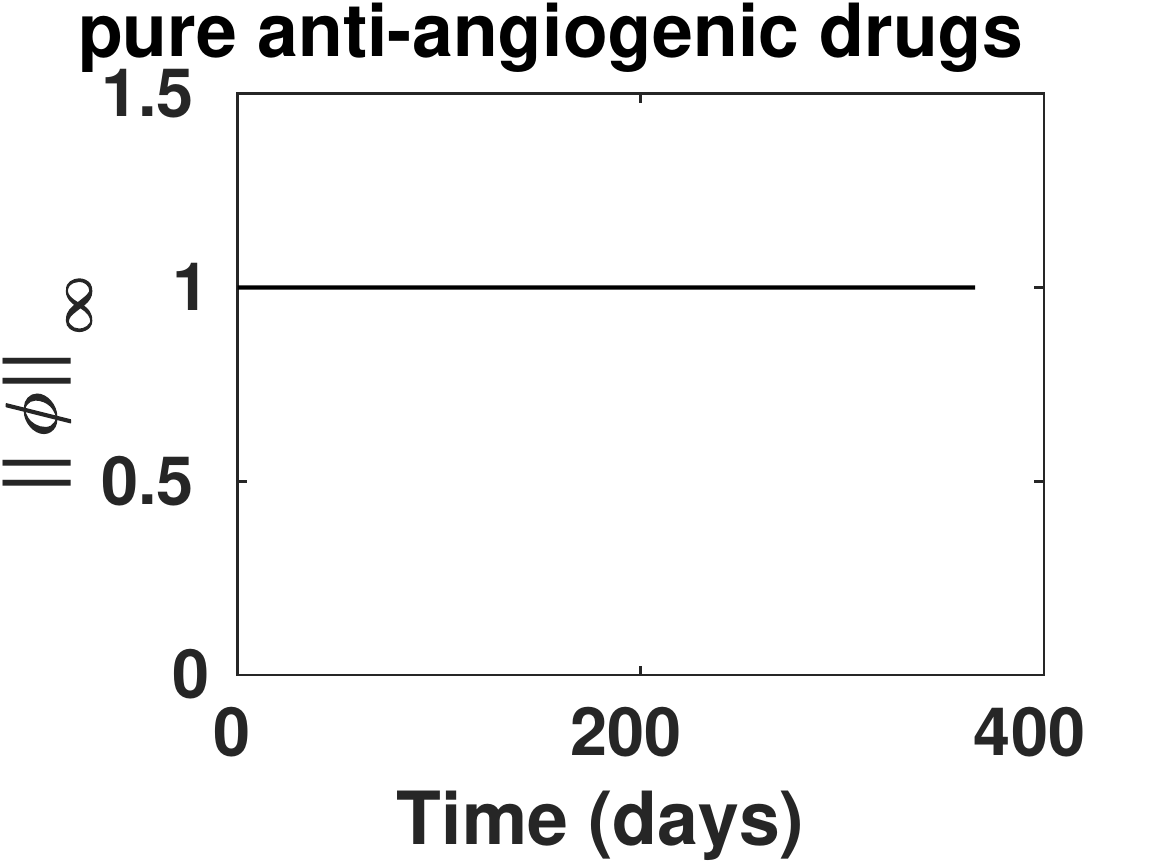}
    \includegraphics[width=0.3\textwidth]{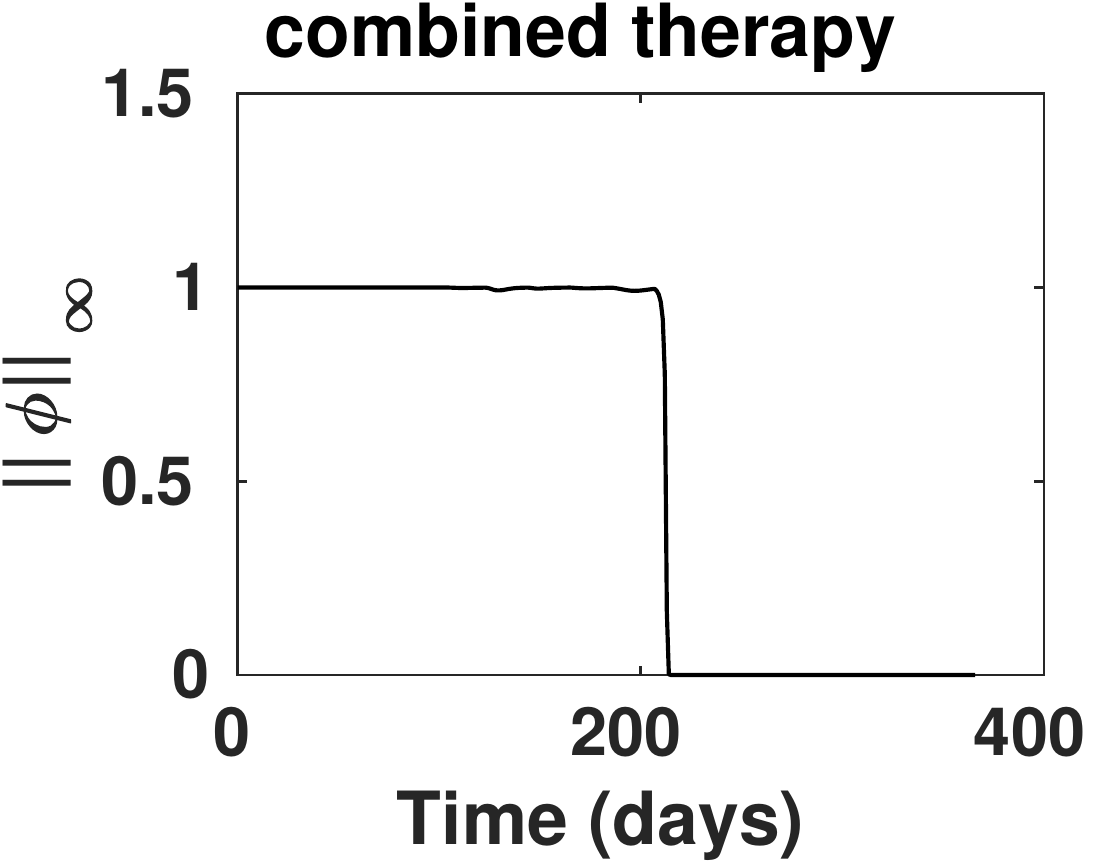}
    \caption{Evaluation of MBP preservation for an aggressive tumor under treatment plans (I–III), using baseline parameters \(S_c = 2.75\), \(\gamma_c = 17\).}
    \label{fig:epsimage}
\end{figure}

\section{Concluding Remarks}  \label{sec:conclusion}

This paper develops a family of matrix-free stabilized BDF schemes for semilinear parabolic equations that satisfy maximum bound principle and an energy dissipation law. We provide a unified theoretical analysis of the proposed schemes. The accuracy of the sBDF$k$ discretizations is established, and unconditional contractivity of the associated fixed-point mappings is proved. For the Allen–Cahn equation with the double-well potential, we further derive an unconditional discrete energy stability result. The resulting algorithms are implemented through matrix-free fixed-point iterations, thereby avoiding the assembly and inversion of large matrices and improving computational efficiency. Numerical experiments validate the theoretical findings and demonstrate both the accuracy and the efficiency of the schemes. 

The present study employs a second-order accurate finite-difference discretization in space. A natural direction for future work is to incorporate sBDF$k$ methods with high-order spatial discretizations that retain the unconditional MBP-preserving and discrete energy stability properties, along with efficient matrix-free implementation.

\section*{Acknowledgements}
The work is supported in part by the National Science Foundation under Grant DMS-2143739 and the ACCESS program through allocation MTH210005.

\bibliographystyle{plain}
\bibliography{ref2}
\end{document}